\documentclass[11pt]{amsart}
\usepackage{graphicx}
\usepackage{amscd}
\usepackage{amsmath}
\usepackage{amsxtra}
\usepackage{amsfonts}
\usepackage{amssymb}
\usepackage{xcolor}
\usepackage{mathrsfs}  
\usepackage{leftindex}
\usepackage[euler]{textgreek}

\newtheorem{theorem}{Theorem}[section]
\newtheorem{corollary}[theorem]{Corollary}
\newtheorem{lemma}[theorem]{Lemma}
\newtheorem{proposition}[theorem]{Proposition}

\newtheorem{conjecture}[theorem]{Conjecture}
\theoremstyle{definition}
\newtheorem{definition}[theorem]{Definition}

\newtheorem{remark}[theorem]{Remark}

\theoremstyle{remark}

\renewcommand{\theclaim}{\textup{\theclaim}}

\newtheorem*{acknowledgements}{Acknowledgements}

\numberwithin{equation}{section}

\def\openone%{\hbox{\upshape \small1\kern-3.3pt\normalsize1}}

{\mathchoice
	
	{\hbox{\upshape \small1\kern-3.3pt\normalsize1}}
	
	{\hbox{\upshape \small1\kern-3.3pt\normalsize1}}
	
	{\hbox{\upshape \tiny1\kern-2.3pt\SMALL1}}
	
	{\hbox{\upshape \Tiny1\kern-2pt\tiny1}}}

\makeatletter

\newbox\ipbox

\newcommand{\ip}[2]{\left\langle #1\, , \,#2\right\rangle}
\newcommand{\diracb}[1]{\left\langle #1\mathrel{\mathchoice
		
		{\setbox\ipbox=\hbox{$\displaystyle \left\langle\mathstrut
				#1\right.$}
			
			\vrule height\ht\ipbox width0.25pt depth\dp\ipbox}
		
		{\setbox\ipbox=\hbox{$\textstyle \left\langle\mathstrut
				#1\right.$}
			
			\vrule height\ht\ipbox width0.25pt depth\dp\ipbox}
		
		{\setbox\ipbox=\hbox{$\scriptstyle \left\langle\mathstrut
				#1\right.$}
			
			\vrule height\ht\ipbox width0.25pt depth\dp\ipbox}
		
		{\setbox\ipbox=\hbox{$\scriptscriptstyle \left\langle\mathstrut
				#1\right.$}
			
			\vrule height\ht\ipbox width0.25pt depth\dp\ipbox}
		
	}\right. }

\newcommand{\dirack}[1]{\left. \mathrel{\mathchoice
		
		{\setbox\ipbox=\hbox{$\displaystyle \left.\mathstrut
				#1\right\rangle$}
			
			\vrule height\ht\ipbox width0.25pt depth\dp\ipbox}
		
		{\setbox\ipbox=\hbox{$\textstyle \left.\mathstrut
				#1\right\rangle$}
			
			\vrule height\ht\ipbox width0.25pt depth\dp\ipbox}
		
		{\setbox\ipbox=\hbox{$\scriptstyle \left.\mathstrut
				#1\right\rangle$}
			
			\vrule height\ht\ipbox width0.25pt depth\dp\ipbox}
		
		{\setbox\ipbox=\hbox{$\scriptscriptstyle \left.\mathstrut
				#1\right\rangle$}
			
			\vrule height\ht\ipbox width0.25pt depth\dp\ipbox}
		
	} #1\right\rangle}

\newcommand{\beq}{\begin{equation}}
	
	\newcommand{\eeq}{\end{equation}}

\newcommand{\cj}[1]{\overline{#1}}

\newcommand{\bz}{\mathbb{Z}}

\newcommand{\br}{\mathbb{R}}
\newcommand{\bc}{\mathbb{C}}

\newcommand{\bn}{\mathbb{N}}

\def\blfootnote{\xdef\@thefnmark{}\@footnotetext}

\input xy
\xyoption{all}
\usepackage{amssymb}

\def\-{^{-1}}

\def\D{\mathscr{D}}

\def\ty{\emptyset}

	\newcommand{\End}{\operatorname*{End}}
\newcommand{\Ends}{\operatorname*{Ends}}
\newcommand{\lmin}{l_{\operatorname*{min}}}

\begin{document}
	
	\title[Spectral properties of unions of intervals and groups of local translations]{Spectral properties of unions of intervals and groups of local translations}

	\author{Bryan Ducasse}
	\address{[Bryan Ducasse] Department of Mathematics\\
		University of Iowa\\
		2 West Washington Street\\
		Iowa City, Iowa 52240, USA\\} \email{bducasse@uiowa.edu}

	\author{Dorin Ervin Dutkay}
	\address{[Dorin Ervin Dutkay] University of Central Florida\\
		Department of Mathematics\\
		4000 Central Florida Blvd.\\
		P.O. Box 161364\\
		Orlando, FL 32816-1364\\
		U.S.A.\\} \email{Dorin.Dutkay@ucf.edu}
	%\thanks{$^*$Corresponding author.}
	
		\author{Colby Fernandez}
	\address{[Colby Fernandez] School of Mathematics\\
		Georgia Institute of Technology\\
		686 Cherry Street\\
		Atlanta, GA 30332-0160\\ USA\\} \email{cfernandez68@gatech.edu}


	\begin{abstract}
		In connection to the Fuglede conjecture, and to Fuglede's original work \cite{Fug74}, 
	we study one-parameter unitary groups associated to self-adjoint extensions of the differential operator $Df=\frac1{2\pi i}f'$ on a union of finite intervals. We present a formula for such unitary groups and we use it to discover some geometric properties of such sets in $\br$ which admit orthogonal bases of exponential functions (also called spectral sets). In particular, we show that, when a union of intervals is a spectral set, then the lengths of the gaps between the intervals must be sums of lengths of the intervals of the set, with possible repetition.

	\end{abstract}
	\maketitle
	\tableofcontents
	\newcommand{\Ds}{\mathsf{D}}
	\newcommand{\Dmax}{\mathscr D_{\operatorname*{max}}}
	\newcommand{\Dmin}{\Ds_{\operatorname*{min}}}
	\newcommand{\dom}{\operatorname*{dom}}
	\newcommand{\Exp}{\operatorname{Exp}}

	\section{Introduction}\label{sec1}

	{\bf Background.} 	In 1958 Irving Segal posed the following problem to Bent Fuglede: Let $\Omega$ be an open set in $\br^d$. Consider the partial differential operators $D_1=\frac1{2\pi i}\frac\partial{\partial x_1},\dots,D_d=\frac1{2\pi i}\frac\partial{\partial x_d}$ defined on the space of $C^\infty$-differentiable, compactly supported functions on $\Omega$, denoted by $C_0^\infty(\Omega)$. With commutation being understood in the sense of commuting spectral measures, under what conditions do
	these partial differential operators admit commuting (unbounded) self-adjoint extensions $H_1,\dots,H_d$ on $L^2(\Omega)$?

	In his 1974 seminal paper \cite{Fug74}, Fuglede offered the following answer to Segal's question, for {\it connected domains} $\Omega$, of {\it finite measure} and which are Nikodym domains, meaning that the Poincar\'e inequality is satisfied:
	
	\begin{theorem}\label{th2.1}\cite[Theorem I]{Fug74}	Let $\Omega\subset\br^d$ be a finite measure, open and connected Nikodym region. Denote by 
		$$e_\lambda(x)=e^{2\pi i\lambda\cdot x},\quad (x\in\br^d,\lambda\in\br^d).$$
		
		\begin{enumerate}
			\item[(a)] Let $H=(H_1,\dots,H_d)$ denote a commuting family (if any) of self-adjoint extensions $H_j$ of the differential operators $D_j$ on $L^2(\Omega)$, $j=1,\dots,n$. Then $H$ has a discrete spectrum, each point $\lambda\in \sigma(H)$ being a simple eigenvalue for $H$ with the eigenspace $\bc e_\lambda$, and hence $(e_\lambda)_{\lambda\in \sigma(H)}$ is an orthogonal basis for $L^2(\Omega)$. Moreover, the spectrum $\sigma(H)$ and the point spectrum $\sigma_p(H)$ are equal and  $$\sigma(H)=\sigma_p(H)=\{\lambda\in\br^d : e_\lambda\in\mathscr D( H)\}.$$
			$\mathscr D(H)$ denotes the intersection of the domains of the operators $H_j$: $\mathscr D(H):=\cap_j \mathscr D(H_j)$.
			\item[(b)] Conversely, let $\Lambda$ denote a subset (if any) of $\br^d$ such that $(e_\lambda)_{\lambda\in\Lambda}$ is an orthogonal basis for $L^2(\Omega)$. Then there exists a unique commuting family $H=(H_1,\dots,H_d)$ of self-adjoint extensions $H_j$ of $D_j$ on $L^2(\Omega)$ with the property that $\{e_\lambda : \lambda\in\Lambda\}\subset\mathscr D(H)$, or equivalently that $\Lambda=\sigma(H)$.
		\end{enumerate}
	\end{theorem}
	
	\begin{definition}
		\label{defsp}
		A subset $\Omega$ of $\br^d$ of finite Lebesgue measure is called {\it spectral} if there exists a set $\Lambda$ in $\br^d$, such that the family of exponential functions 
		$$\{e_\lambda : \lambda\in\Lambda\}$$
		forms an orthogonal basis for $L^2(\Omega)$. In this case, we say that $\Lambda$ is a {\it spectrum} for $\Omega$. 
		
		We say that a set $\Omega$ {\it tiles} $\br^d$ by translations if  there exists a set $\mathcal T$ in $\br^d$ such that $\{\Omega +t : t\in\mathcal T\}$ is a partition of $\br^d$ up to measure zero. We also call $\mathcal T$ a {\it tiling set} for $\Omega$ and we say that $\Omega$ {\it tiles by} $\mathcal T$. 
	\end{definition}
	
	In these terms, when $\Omega$ is connected, of finite measure and a Nikodym domain, there exist commuting self-adjoint extensions of the partial differential operators $\{D_j\}$ if and only if $\Omega$ is spectral.  Fuglede proposed a geometric characterization of spectral sets in his famous conjecture:
	\begin{conjecture}{\bf[Fuglede's Conjecture]}
		\label{conFu}
		A subset $\Omega$ of $\br^d$  of finite Lebesgue measure is spectral if and only if it tiles $\br^d$ by translations.
	\end{conjecture}
	The conjecture was proved to be false in dimensions 3 and higher \cite{Tao04, FMM06}, but it is true for convex domains \cite{LM22}.

	In this paper we focus on dimension one, on finite unions of finite intervals. Throughout the paper $\Omega$ will denote such a union:
	\begin{equation}
		\label{eqom}
			\Omega=\bigcup_{i=1}^n(\alpha_i,\beta_i),\mbox{ where }-\infty<\alpha_1<\beta_1<\alpha_2<\beta_2<\dots\alpha_n<\beta_n<\infty.
	\end{equation}

{\bf Previous Results.}	The self-adjoint extensions of the differential operator $D=D_1$ on $C_0^\infty(\Omega)$ were determined in \cite{DJ15b, DJ23, CDJ25}; 	detailed analyses can be found for unions of two finite intervals in \cite{JPT12,JPT13}, for finite unions of finite intervals in \cite{DJ15b, DJ23, CDJ25}, for finite unions of intervals the form $$\Omega=(-\infty,\beta_1)\cup\bigcup_{j=1}^n(\alpha_j,\beta_j)\cup (\alpha_{n+1},\infty)$$ in \cite{JPT15}. 
	
	We recall some of the main results in \cite{DJ15b,DJ23}. Define the maximal domain 
	$$\mathscr D=\{f\in L^2(\Omega) : \mbox{ The weak/distributional derivative $f'$ is in $L^2(\Omega)$}\}.$$
	It can also be described as 
	$$\mathscr D=\{f\in L^2(\Omega): f \mbox{ absolutely continuous on each interval and }f'\in L^2(\Omega)\}.$$
	
	We use the same notation $D$ for the differential operator on $\mathscr D$, $Df=\frac1{2\pi i}f'$, $f\in\mathscr D$.

	For an absolutely continuous function $f\in\mathscr D$, the side limits $f(\alpha_i+)$ and $f(\beta_i-)$ are well defined, and we denote them by $f(\alpha_i)$ and $f(\beta_i)$, respectively, $i=1,\dots,n$.
	We also use the notation 
	$f(\vec\alpha)=(f(\alpha_1),\dots,f(\alpha_n))$, and similarly for $f(\vec\beta)$. 
	
	For $\vec z=(z_1,z_2,\dots,z_n)\in\bc^n$ denote by $E(\vec z)$, the $n\times n$ diagonal matrix with entries $e^{2\pi iz_1}$, $e^{2\pi iz_2}$, $\dots,e^{2\pi iz_n}$.

	The self-adjoint extensions $H$ of the differential operator $D$ are specified by a certain $n\times n$ ``boundary'' matrix $B$ which relates the boundary values of the functions in the domain by $Bf(\vec\alpha)=f(\vec\beta)$.
	
	\begin{theorem}\label{th7.1}\cite[Theorem 3.2, Theorem 4.2]{DJ23}
		A self-adjoint extension $H$ of the maximal operator $D$ on $\Omega$ is uniquely determined by an $n\times n$ unitary matrix $B$ (which we call {\it the boundary matrix associated to $H$}), through the condition 
		$$\mathscr D(H)=\{f\in\mathscr D : Bf(\vec\alpha)=f(\vec\beta)\}=:\mathscr{D}_B.$$
		Conversely, any unitary $n\times n$ matrix $B$ determines a self-adjoint extension $H$ of $D$ with domain $\mathscr D(H)$ as above. We denote $H=D_B$.

		The spectral measure of $D_B$ is atomic, supported on the spectrum  
		$$\Lambda:=\sigma(D_B)=\left\{\lambda\in\bc : \det(I-E(\lambda\vec\beta)^{-1}BE(\lambda\vec\alpha))=0\right\}\subseteq\br$$
		which is a discrete unbounded set. For an eigenvalue $\lambda\in\sigma(D_B)$, the eigenspace has dimension at most $n$, and it consists of functions of the form 
		
		\begin{equation}\label{eq7.1.1}
			f(x)= e^{2\pi  i \lambda x} \sum_{i=1}^n c_i\chi_{(\alpha_i,\beta_i)}(x),\mbox{ where }c=(c_i)_{i=1}^n\in\bc^n\mbox{, and }BE(\lambda\vec\alpha)c=E(\lambda\vec\beta)c.
		\end{equation}
		
	\end{theorem}

	Next, given the self-adjoint extension $D_B$ with boundary matrix $B$, one can define the {\it associated unitary group} $\{U(t)\}_{t\in\br}=\{U_B(t)\}_{t\in\br}$ by
	\begin{equation}
		\label{equb}
		U(t)=\exp(2\pi it D_B),\quad(t\in\br).
	\end{equation}

	 This unitary group acts as translation inside the intervals (see point (ii) in the next theorem), and once it reaches a boundary point, it splits with probabilities defined by the matrix $B$ (point (iii) in the next theorem).

	\begin{theorem}\label{thdju}\cite[Theorem 5.1]{DJ23}
		Let $D_B$ be a self-adjoint extension of the differential operator $D$, with boundary matrix $B$. Let $$U(t)=\exp{(2\pi i t D_B)},\quad (t\in\br),$$ be the associated one-parameter unitary group.
		\begin{enumerate}
			\item The domain $\D_B$ is invariant for $U(t)$ for all $t\in\br$, i.e., if $f\in\mathscr D$ with $B f(\vec\alpha)= f(\vec\beta)$, then $U(t)f\in\mathscr D$ with $B(U(t)f)(\vec\alpha)=(U(t)f)(\vec\beta)$.
			\item Fix $i\in\{1,\dots,n\}$ and let $t\in\br$ such that $(\alpha_i,\beta_i)\cap ((\alpha_i,\beta_i)-t)\neq \ty$. Then, for $f\in L^2(\Omega)$, 
			
			\begin{equation}
				(U(t) f)(x)=f(x+t),\mbox{ for a.e. }x\in (\alpha_i,\beta_i)\cap ((\alpha_i,\beta_i)-t).
				\label{equ1}
			\end{equation}

			In particular,
			\begin{equation}
				(U(\beta_i-x) f)(x)=f(\beta_i),\, (U(\alpha_i-x)f)(x)=f(\alpha_i),\mbox{ for }f\in\D_B, x\in (\alpha_i,\beta_i).
				\label{eq4.2.2}
			\end{equation}
			\item For $f\in\mathscr D_B$, if $x\in (\alpha_i,\beta_i)$ and $t>\beta_i-x$, then
			
			\begin{equation}
				\left[ U(t)f\right](x)=\pi_i\left(B\left[U(t-(\beta_i-x))f\right](\vec\alpha)\right)=\sum_{k=1}^n b_{i,k}[U(t-(\beta_i-x))f](\alpha_i).
				\label{eq4.4.1}
			\end{equation}
			Here $\pi_i:\bc^n\rightarrow\bc$ denotes the projection onto the $i$-th component, i.e., $\pi (x_1,x_2\dots,x_n)=x_i$.
			
		\end{enumerate}
	\end{theorem}
	
	The next question is when is a union of intervals spectral? The condition is that the constants $c_i$ that define the eigenvectors in \eqref{eq7.1.1} are all equal, so that the eigenvectors of $D$ in the domain $\mathscr D_B$ are exponential functions. 
	\begin{theorem}\label{th7.3}\cite[Theorem 6.5, Theorem 6.7]{DJ23}
		The union of intervals $\Omega$ is spectral if and only if there exists an $n\times n$ unitary matrix $B$ with the property that for each $\lambda\in\br$, the equation $BE(\lambda\vec\alpha)c=E(\lambda\vec\beta)c$, $c\in\bc^n$ has either only the trivial solution $c=0$, or only constant solutions of the form $c=a(1,1,\dots,1)$, $a\in\bc$. In this case a spectrum $\Lambda$ of $\Omega$ is the spectrum of the self-adjoint extension $D_B$ associated to the boundary matrix $B$, as in Theorem \ref{th7.1}.
		
		If $\Lambda$ is a spectrum for $\Omega$, then the matrix $B$ is uniquely determined and well-defined by the conditions 
		\begin{equation}
			Be_\lambda(\vec \alpha)=e_\lambda(\vec \beta),\mbox{ for all }\lambda\in\Lambda.
			\label{eqfu4.1}
		\end{equation}
		Moreover 
		\begin{equation}
			\textup{span}\{ e_\lambda(\vec \alpha) : \lambda\in\Lambda\}=\textup{span} \{e_\lambda(\vec\beta):\lambda\in\Lambda\}=\bc^n.
			\label{eqfu4.2}
		\end{equation}
		
		Conversely, if the boundary matrix $B$ has the specified property then a spectrum for $\Omega$ is given by 
		\begin{equation}
			\Lambda=\{\lambda\in\br : Be_\lambda(\vec\alpha)=e_\lambda(\vec\beta)\}.
			\label{eqfu4.3}
		\end{equation}
	\end{theorem}
	
One of the nice properties of the unitary group $\{U(t)\}$, that we will use in this paper, is that, when the set $\Omega$ is spectral, the unitary group $\{U(t)\}$ acts as translations not only inside intervals, but also when it jumps from one interval to another. More precisely, for any $t\in\br$, and $f\in L^2(\Omega)$,
	$$(U(t)f)(x)=f(x+t)\mbox{ for a.e. }x\in\Omega\cap(\Omega-t).$$
	In other words $(U(t)f)(x)=f(x+t)$ when both $x$ and $x+t$ are in $\Omega$. We call such a unitary group, a {\it group of local translations}.

	\begin{theorem}\label{th7.4}\cite[Theorem 6.8]{DJ23}
		Assume that $\Omega$ is spectral with spectrum $\Lambda$. Then the unitary group $U=U_\Lambda$ associated to $\Lambda$ is a unitary group of local translations. 
		Conversely, if there exists a strongly continuous unitary group of local translations $(U(t))_{t\in\br}$ for $\Omega$, then $\Omega$ is spectral. 
	\end{theorem}

	{\bf Notations.} Throughout the paper $\Omega$ is the union of intervals in \eqref{eqom}. We denote by $l_i=\beta_i-\alpha_i$, the length of the interval $(\alpha_i,\beta_i)$ of $\Omega$, $1\leq i\leq n$. We denote by $\mathfrak m$ the Lebesgue measure on $\br$. 
	
	If a unitary boundary matrix  $B$ is given, $D_B$ represents the self-adjoint extension of $D$ with domain $\mathscr D_B$, as in Theorem \ref{th7.1} and $\{U(t)\}_{t\in\br}$ is the unitary group associated to it as in \eqref{equb}. $\Lambda$ denotes the spectrum of $D_B$ as in Theorem \ref{th7.1}. 
	
	In the case when $\Omega$ is spectral, $\Lambda$ will denote its spectrum, and $B$ is the unitary boundary matrix associated to it as in Theorem \ref{th7.3}.

	\medskip
	
	{\bf New Contributions.} 	The paper is organized as follows: in Section \ref{sec2}, we answer a question posed in \cite[Section 6 - ``Questions'']{JPT13}, by presenting an explicit formula for the unitary group $\{U(t)\}$ (Theorem \ref{thu2}).

	In Section 3, we study the case when $\Omega$ is a spectral set and we use the formula for $\{U(t)\}$ to determine some geometric properties of $\Omega$. The central result of Section 3 is Theorem \ref{thl1}, which states (among other things), that, for spectral sets, the gaps between the intervals of $\Omega$ have to be sums of lengths of the intervals of $\Omega$. The formulas in Theorem \ref{thl1}(ii) provide some connections between the geometry of the set $\Omega$, encoded in the lengths of its intervals and of its gaps, and the spectral data of the set $\Omega$, encoded in the boundary matrix $B$. 
	
	In Section \ref{sec4}, we consider a restriction that was suggested by Segal to Fuglede: what if the unitary group $\{U(t)\}$ is multiplicative (Definition \ref{defm1})? Fuglede showed in \cite{Fug74}, that (even in higher dimension) when $\Omega$ is connected and spectral and the unitary group $\{U(t)\}$ is multiplicative, then $\Omega$ has spectrum a lattice and it also tiles $\br^d$ by a lattice. We show here that similar statements hold for $\Omega$ a union of intervals: $\{U(t)\}$ is multiplicative if and only if $B$ is a permutation matrix (Theorem \ref{thm4}); in the case when $\Omega$ is spectral and $\{U(t)\}$ is multiplicative, the permutation is a full cycle, $\Omega$ tiles with $\mathfrak m(\Omega)\bz$, and one can read from the permutation, how the set $\Omega$ is translation congruent to an interval with the same spectrum (Theorem \ref{thm5}). At the end of Section \ref{sec4} we prove similar results, when the multiplicative condition is replaced by the Forelli condition (Definition \ref{defo}).

	\section{The formula for $U(t)$}\label{sec2}

	We begin with a simple proposition which shows how our elements are modified when one applies the reflection $t\to-t$. This will help us reduce problems to the case $t$ positive, as we will describe a little later in Remark \ref{reminu}.
	
	\begin{proposition}\label{prminu}
		
		Let $D_B$ and $\{U_B(t)\}$ be the self-adjoint extension and the unitary group associated to the set $\Omega$ and the boundary matrix $B$,  and 	let $D_{B^*}$ and $\{U_{B^*}(t)\}$ be the self-adjoint extension and the unitary group associated to the set $-\Omega$ and the boundary matrix $B^*$. Let $J:L^2(\Omega)\to L^2(-\Omega)$, be the unitary operator
		$$Jf(x)=f(-x),\quad (f\in L^2(\Omega), x\in -\Omega).$$
		\begin{enumerate}
			\item $JD_B=-D_{B^*}J$ and $J\mathscr D(D_B)=\mathscr D(D_{B^*})$. 
			\item $U_{B^*}(t)=JU_B(-t)J^*$ for all $t\in\br$. 
			\item $\{U_B(t)\}$ is a group of local translations if and only if $\{U_{B^*}(t)\}$ is. 
		\end{enumerate}
	
	\end{proposition}
	
	\begin{proof}
It is clear that $J$ is unitary and $J^*f(x)=f(-x)$ for $f\in L^2(-\Omega)$, $x\in\Omega$. 

	For (i), a function $f\in L^2(-\Omega)$ is absolutely continuous with derivative in $L^2(-\Omega)$ if and only if $J^*f$ has the same properties relative to $\Omega$. We check the boundary conditions: $B^*f(-\vec\beta)=f(-\vec\alpha)$ iff $B^*J^*f(\vec\beta)=J^*f(\vec\alpha)$ iff $BJ^*f(\vec\alpha)=J^*f(\vec\beta)$. Thus $\mathscr D(D_{B^*})=J\mathscr{D}(D_B)$. 
	
	For $f\in\mathscr D(D_B)$, $x\in- \Omega$ 
	$$D_{B^*}Jf(x)=\frac1{2\pi i}(f(-x))'=-\frac1{2\pi i}f'(-x)=-JD_Bf(x),$$
	and (i) follows. 
	
	For (ii), since $D_{B^*}=J(-D_B)J^*$, we obtain, by functional calculus, 
	$$U_{B^*}(t)=\exp(2\pi it D_{B^*})=J\exp(-2\pi itD_B)J^*=JU_B(-t)J^*.$$
	
	For (iii), suppose $\{U_B(t)\}$ is a group of local translations and let $x\in-\Omega$ and $t\in\br$ such that $x+t\in-\Omega$ and let $f\in L^2(-\Omega)$. Then $-x$ and $-x-t$ are in $\Omega$ so $U_B(-t)J^*f(-x)=J^*f(-x-t)$; this can be rewritten as $J^*U_{B^*}(t)f(-x)=J^*f(-x-t)$ so $U_{B^*}(t)f(x)=f(x+t)$. Thus $\{U_{B^*}(t)\}$ is a group of local translations. The converse is analogous. 
	\end{proof}
	
	\begin{remark}\label{reminu}
		Proposition \ref{prminu} can be used to deduce results for $U_B(t)$ with $t$ negative, once we have the results for $U_B(t)$ with $t$ positive. The idea is to take $t$ negative and apply the results for the set $-\Omega$, the boundary matrix $B^*$ and $-t$; then we obtain some results or formulas for $U_{B^*}(-t)$; but, with Proposition \ref{prminu}(ii), we obtain results/formulas for $U_B(t)$. 
		
	\end{remark}

In the case when all the lengths of the intervals of $\Omega$ are the same, we can also say a bit more about the self-adjoint extension $D_B$, and about the matrix $B$ when $\Omega$ is spectral. We include this in the next proposition.
\begin{proposition}\label{prm1}
	Suppose the intervals of $\Omega$ have the same length $l_k={\ell}$ for all $1\leq k\leq n$. Let $B$ be a unitary boundary matrix and let $D_B$ be the self-adjoint extension of $D$ associated to the matrix $B$. 
	
	\begin{enumerate}
		\item A point $\lambda$ is in the spectrum $\sigma(D_B)$ of $D_B$ if and only if $e^{2\pi i\lambda {\ell}}$ is in the spectrum $\sigma(B)$ of $B$. Therefore, if $\sigma(B)=\{e^{2\pi i\theta_1},\dots, e^{2\pi i\theta_n}\}$ then 
		\begin{equation}
			\label{eqm1.1}
			\Lambda:=\sigma(D_B)=\frac1{\ell}\left(\left\{\theta_1,\dots,\theta_n\right\}+\bz\right).
		\end{equation}
		\item For a vector $\vec c\in\bc^n$, $BE(\lambda\vec\alpha)\vec c=E(\lambda\vec\beta)\vec c$ if and only if $E(\lambda\vec\alpha)\vec c$ is an eigenvector for $B$ with eigenvalue $e^{2\pi i\lambda {\ell}}$. 
		\item Assume in addition that $B$ is spectral. Then all the eigenvalues of $B$ have multiplicity one and, if $e^{2\pi i \lambda {\ell}}$, $\lambda\in \br$ is an eigenvalue for $B$, then its eigenvector (unique up to a multiplicative constant) is $(e^{2\pi i\lambda\alpha_1},\dots,e^{2\pi i\lambda\alpha_n})^T$.
	\end{enumerate}
	
\end{proposition}
\begin{proof}
	Since the intervals have equal length ${\ell}$, we have $\beta_k=\alpha_k+{\ell}$, for all $1\leq k\leq n$. By Theorem \ref{th7.1}, we have that $\lambda\in\sigma(D_B)$ if and only if there is a nonzero vector $\vec c\in\bc^n$ such that $BE(\lambda\vec \alpha)\vec c=E(\lambda\vec\beta)\vec c$. But $E(\lambda\vec\beta)=e^{2\pi i\lambda {\ell}}E(\lambda\vec\alpha)$, therefore we can rewrite the condition as  $BE(\lambda\vec \alpha)\vec c=e^{2\pi i\lambda {\ell}}E(\lambda\vec \alpha)\vec c$. This means that $E(\lambda\vec \alpha)\vec c$ should be an eigenvector for $B$ with eigenvalue $e^{2\pi i\lambda {\ell}}$. This proves (i) and (ii). 
	
	Assume now, in addition, that $B$ is spectral. Let $\vec x\neq 0$ be an eigenvector for $B$ with eigenvalue $e^{2\pi i\lambda {\ell}}$, $\lambda\in\br$. Then, using (ii), $\vec c:=E(-\lambda\vec\alpha)\vec x$ satisfies the equation  $BE(\lambda\vec\alpha)\vec c=E(\lambda\vec\beta)\vec c$. Since $\Omega$ is spectral, with Theorem \ref{th7.3}, $\vec c$ has to have equal entries $c$. Therefore $x=E(\lambda\vec\alpha)\vec c=c\left(e^{2\pi i\lambda\alpha_1},\dots, e^{2\pi i\lambda\alpha_n}\right)$. 
\end{proof}

Next, we will present the formula for the unitary group $\{U(t)\}$ associated to the union of intervals $\Omega$ and the unitary boundary matrix $B$. Before that, we want to give the reader some intuition. Imagine a point $x$ in $\Omega$, say in the interval $(\alpha_i,\beta_i)$, at time $t=0$; we denote $i_1=i$. It has to travel for $T>0$ seconds total. It starts to move to the right, inside the interval $(\alpha_i,\beta_i)$, after $t$ seconds it is at the point $x+t$, until it reaches the endpoint $\beta_i$ (if $T>\beta_i-x$); at that moment, $\beta_i-x$ seconds have past. Once it goes trough $\beta_i$ it splits into $n$ points, one in each interval $i_2$, $1\leq i_2\leq n$, at the left-endpoint $\alpha_{i_2}$, with a weight $b_{i_1,i_2}$; the numbers $|b_{i_1,i_2}|^2$, $1\leq i_2\leq n$ can be thought of as probabilities for the point at $\beta_{i_1}$ to jump into the point at $\alpha_{i_2}$, since $$\sum_{i_2=1}^n |b_{i_1,i_2}|^2=1.$$
Now each point continues the journey from $\alpha_{i_2}$. There are $T-(\beta_i-x)$ seconds left until time expires. If $T-(\beta_i-x)$ is less than the length of the interval $l_{i_2}=\beta_{i_2}-\alpha_{i_2}$, then the point will stop at the point $\alpha_{i_2}+T-(\beta_i-x)$, and we denote $T-(\beta_i-x)=r(x,t,i_1i_2)$, and we have $0\leq r(x,t,i_2)<l_{i_2}$. It will have weight $b_{i_1,i_2}$. 

If $T-(\beta_i-x)\geq l_{i_2}$, then the point reaches $\beta_{i_2}$, after an extra $l_{i_2}$ seconds and it splits again into $n$ points at $\alpha_{i_3}$, $1\leq i_3\leq n$, with weights $b_{i_2,i_3}$, so the combined weight will be $b_{i_1,i_2}b_{i_2,i_3}$. At the moment of the second split $(\beta_i-x)+l_{i_2}$ seconds have passed, and so $T-((\beta_i-x)+l_{i_2})$ seconds are remaining.

The process continues until the time expires. Thus, one has to split the interval of time $T$ into pieces: the first piece is $\beta_i-x$ required for the point to get out of the interval $(\alpha_i,\beta_i)$, then the next pieces will be the lengths of the intervals: $l_{i_2}$, $l_{i_3}$, \dots, $l_{i_{m-1}}$ so that $T>(\beta_i-x)+l_{i_2}+l_{i_3}+\dots +l_{i_{m-1}}$. The last interval has length $l_{i_m}$ smaller than the remainder of time $r:=T-((\beta_i-x)+l_{i_2}+l_{i_3}+\dots +l_{i_{m-1}})$. So, on this path, the point stops at $\alpha_{i_m}+r$ in the interval $(\alpha_{i_m},\beta_{i_m})$. It will carry the weight $b_\omega=b_{i_1,i_2}b_{i_2,i_3}\dots b_{i_{m-1},i_m}$.

With this intuition, we construct the following definition:

\begin{definition}\label{defu1}{\bf [$t$ positive]}
	Let $x\in \Omega$, $x\in (\alpha_i,\beta_i)$ and $t>0$. An {\it admissible path} for $x$ and $t$ is a sequence of indices $\omega=i_1i_2\dots i_m$, $i_k\in\{1,\dots,n\}$ for all $1\leq k\leq m$, with the following properties:
	\begin{enumerate}
		\item $i_1=i$. 
		\item If $t<\beta_i-x$ then $m=1$
		\item If $t\geq \beta_i-x$ then $m\geq 2$ and 
		$$(\beta_i-x)+l_{i_2}+l_{i_3}+\dots+l_{i_{m-1}}\leq t<(\beta_i-x)+l_{i_2}+l_{i_3}+\dots+l_{i_{m-1}}+l_{i_m}.$$
	\end{enumerate}
	If $t>\beta_i-x$, we denote by 
	\begin{equation}
		\label{equ1.1}
		r(x,t,\omega):=t-\left((\beta_i-x)+l_{i_2}+l_{i_3}+\dots+l_{i_{m-1}}\right),
	\end{equation}
	and note that $0\leq r(x,t,\omega)<l_{i_m}$. We call $r(x,t,\omega)$ the {\it remainder} of the path $\omega$. We call $m$ the {\it length} of $\omega$. 
	Also, we denote by 
	\begin{equation}
		\label{equ1.2}
		b_{\omega}:=b_{i_1,i_2}b_{i_2,i_3}\dots b_{i_{m-1},i_m}.
	\end{equation}

	We call $\alpha_{i_m}+r(x,t,\omega)$ the {\it end} of the path $\omega$ (starting at $x$), and we denote it by $\End(x,t,\omega)$.

	Let $\mathcal P_{x,t}$ be the set of all admissible paths for $x$ and $t$, and let $\Ends(x,t)$ be the set of all ends of the admissible paths for $x$ and $t$, 
	$$\Ends(x,t)=\{\End(x,t,\omega): \omega\in\mathcal P_{x,t}\}.$$
	
\end{definition}

\begin{theorem}\label{thu2}
Let $\{U(t)\}$ be the unitary group associated to the union of intervals $\Omega$ and boundary matrix $B$.	Let $f\in\mathscr D_B$ and $t>0$ . Let $x\in\Omega$, $x\in (\alpha_i,\beta_i)$. 
	
	If $t<\beta_i-x$, then 
	\begin{equation}
		\label{equ2.1}
		[U(t)f](x)=f(x+t).
	\end{equation}
	If $t\geq\beta_i-x$,	
	\begin{equation}\label{equ2.2}
			[U(t)f](x)=\sum_{\omega=i_1\dots i_m\in\mathcal P_{x,t}}b_\omega f(\alpha_{i_m}+r(x,t,\omega))=\sum_{\omega\in\mathcal P_{x,t}}b_\omega f(\End(x,t,\omega)).
	\end{equation}
 If $f\in L^2(\Omega)$ and $t>0$ is fixed, then the formulas above hold for almost every $x\in\Omega$.
\end{theorem}

\begin{proof}
First, consider $f\in \mathscr D_B$. 	Let $0<\epsilon<\min_{j} l_j$. Write $t=k\epsilon+r$ with $k\in\bn\cup\{0\}$ and $0\leq r<\epsilon$. We will prove our theorem by induction on $k$. 
	
	Suppose first that $0\leq t\leq \epsilon$. Let $x\in \Omega$, $x\in (\alpha_i,\beta_i)$. We have two cases: $t<\beta_i-x$ and $t\geq \beta_i-x$. 
	
	In the first case, $t<\beta_i-x$ and so $x+t\in(\alpha_i,\beta_i)$, therefore we can use Theorem \ref{thdju} and obtain that 
	$$U(t)f(x)=f(x+t).$$
	
	In the second case, $t\geq\beta_i-x$, and with Theorem \ref{thdju} we have 
	
\begin{equation}\label{equ2.3}
		[U(t)f](x)=\sum_{i_2=1}^n b_{i,i_2}[U(t-(\beta_i-x))f](\alpha_{i_2})=\sum_{i_2=1}^nb_{i,i_2}f(\alpha_{i_2}+ (t-(\beta_i-x))).
\end{equation}
	Note that $0\leq t-(\beta_i-x)<l_{i_2}$ because $t<\epsilon<\min l_j$. Thus the admissible paths $\omega$ for $x$ and $t$, in this case have length $m=2$ and $r(x,t,\omega)=t-(\beta_i-x)$. This shows that the formula for $U(t)$ is true for $0\leq t\leq \epsilon$. This covers the initial step in the induction, $k=0$. 
	
	Assume the formula is true for any $t<k\epsilon$ and take $k\epsilon\leq t<(k+1)\epsilon$. 
	
	Again, take $x\in (\alpha_i,\beta_i)$. We have 
	$$[U(t)f](x)=[U(\epsilon)U(t-\epsilon)f](x).$$
	We know the formula works for $U(\epsilon)$ and for $U(t-\epsilon)$ (since $0\leq t-\epsilon<k\epsilon$). (Note that $\{U(t)\}$ leaves $\mathscr D_B$ invariant (Theorem \ref{thdju})and therefore $U(t-\epsilon)f\in\mathscr D_B$).  We can assume also that $t>\beta_i-x$, otherwise we stay inside the interval $(\alpha_i,\beta_i)$ and the formula is known to be true. 
	
	Suppose first that $\epsilon<\beta_i-x$. Then  $x+\epsilon\in (\alpha_i,\beta_i)$ and 
	
	\begin{equation}
		\label{equ2.4}
			[U(t)f](x)=[U(\epsilon)(U(t-\epsilon)f)] (x)=[U(t-\epsilon)f](x+\epsilon)=\sum_{\omega=i_1\dots i_m\in\mathcal P_{x+\epsilon, t-\epsilon}}b_\omega f(\alpha_{i_m}+r(x+\epsilon,t-\epsilon,\omega)).		
	\end{equation}

	A path $\omega=i_1\dots i_m$ with $i_1=i$ is admissible for $x+\epsilon$ and $t-\epsilon$ if and only if 
	 $$(\beta_i-(x+\epsilon))+l_{i_2}+\dots  +l_{i_{m-1}}\leq t-\epsilon<(\beta_i-(x+\epsilon))+l_{i_2}+\dots  +l_{i_{m-1}}+l_{i_m}.$$
	 A path $\omega=i_1\dots i_m$ with $i_1=i$ is admissible for $x$ and $t$ if and only if 
	  $$(\beta_i-x)+l_{i_2}+\dots  +l_{i_{m-1}}\leq t<(\beta_i-x)+l_{i_2}+\dots  +l_{i_{m-1}}+l_{i_m}.$$
	  Clearly these conditions are the same. Thus $\mathcal P_{x+\epsilon, t-\epsilon}=\mathcal P_{x,t}$.

	  Let's compare the remainders. If $\omega=i_1\dots i_m$ is in $\mathcal P_{x+\epsilon,t-\epsilon}$ then 
	  $$r(x+\epsilon,t-\epsilon,\omega)=t-\epsilon -\left((\beta_i-(t-\epsilon))+l_{i_2}+\dots+l_{i_{m-1}}\right)=r(x,t,\omega).$$
	  
	  Hence, equation \eqref{equ2.4} can be rewritten 
	  $$[U(t)f](x)=\sum_{\omega=i_1\dots i_m\in \mathcal P_{x,t}}b_\omega f(\alpha_{i_m}+r(x,t,\omega)).$$

	  For the second case, $\epsilon\geq \beta_i-x$. Now, we use \eqref{equ2.3} for $\epsilon$ and we have, with $i_1=i$, 
	  $$[U(t)f](x)=[U(\epsilon)(U(t-\epsilon)f)](x)=\sum_{i_2=1}^n b_{i_1, i_2} [U(t-\epsilon) f](\alpha_{i_2}+(\epsilon-(\beta_i-x)))$$
	  \begin{equation}
\label{equ2.5}
 =\sum_{i_2=1}^nb_{i_1,i_2}\sum_{\omega=i_2\dots i_m\in\mathcal P_{\alpha_{i_2}+(\epsilon-(\beta_i-x)), t-\epsilon}}b_\omega f(\alpha_m+r(\alpha_{i_2}+(\epsilon-(\beta_i-x)), t-\epsilon,\omega)).
	  \end{equation}

	   A path $\tilde \omega=i_1i_2\dots i_m$ with $i_1=i$ is admissible for $x$ and $t$ if and only if 
	 $$(\beta_i-x)+l_{i_2}+\dots  +l_{i_{m-1}}\leq t<(\beta_i-x)+l_{i_2}+\dots  +l_{i_{m-1}}+l_{i_m}.$$
	 
	 A path $\omega=i_2\dots i_m$ is admissible for $\alpha_{i_2}+(\epsilon-(\beta_i-x))$ and $t-\epsilon$ if and only if 
	  $$(\beta_{i_2}-(\alpha_{i_2}+(\epsilon-(\beta_i-x))))+l_{i_3}+\dots  +l_{i_{m-1}}\leq t-\epsilon <(\beta_{i_2}-(\alpha_{i_2}+(\epsilon-(\beta_i-x))))+l_{i_3}+\dots  +l_{i_{m-1}}+l_{i_m}.$$
	  
	 Canceling the $\epsilon$'s and using $\beta_{i_2}-\alpha_{i_2}=l_{i_2}$, we see that $\tilde \omega=i_1i_2\dots i_m$ is admissible for $x$ and $t$ if and only if 
	 $\omega=i_2\dots i_m$ is admissible for $\alpha_{i_2}+(\epsilon-(\beta_i-x))$ and $t-\epsilon$.
	 
	 Let's look at the remainder 
	 $$r(\alpha_{i_2}+(\epsilon-(\beta_i-x)), t-\epsilon,\omega)=t-\epsilon-\left(\beta_{i_2}-(\alpha_{i_2}+(\epsilon-(\beta_i-x)))+l_{i_3}+\dots+l_{i_{m-1}}\right)$$
	 $$=t-\left((\beta_i-x)+l_{i_2}+\dots+l_{i_{m-1}}\right)=r(x,t,\tilde\omega).$$
	 
	 Also, note that 
	 $$b_{\tilde \omega}=b_{i_1,i_2}b_\omega.$$
	 Thus, equation \eqref{equ2.5} can be rewritten as 
	 $$[U(t)f](x)=\sum_{\tilde \omega\in \mathcal P_{x,t}}b_{\tilde\omega} f(\alpha_{i_m}+r(x,t,\tilde\omega)).$$
	 This proves that the formula for $U(t)$ works also when $t<\epsilon$. By induction, the conclusion follows.

	 Consider now the case when $f\in L^2(\Omega)$. Fix again $t>0$. We can approximate $f$ in $L^2(\Omega)$ by a sequence of functions $\{f_k\}$ is $C_0^\infty(\Omega)\subseteq \mathscr D_B$. The formulas are true everywhere for the $U(t)f_k$. We have that $\{U(t)f_k\}$ converges in $L^2$ to $U(t)f$. We can extract a subsequence which we relabel $\{f_k\}$ so that $\{f_k\}$ converges pointwise a.e. and in $L^2$ to $f$, and $\{U(t)f_k\}$ converges pointwise a.e. and in $L^2$ to $U(t)f$. Let $N$ be the measure zero set where one of these pointwise convergences does not hold. The set $\mathcal N$ of points $x$ with the property that $x\in N$ or $\End(x,t,\omega)\in N$ for some $\omega\in \mathcal P_{x,t}$ is a finite union of null sets, and therefore it is null. 
	 
	 Indeed, let $\omega=i_1\dots i_m$. If $\omega\in\mathcal P_{x,t}$ and $\End(x,t,\omega)\in N$, then $\alpha_j+r(x,t,\omega)\in N$. But
	 $$r(x,t,\omega)=t-(\beta_i-x+l_{i_2}+\dots+l_{i_{m-1}}).$$
	 Then $$\End(x,t,\omega)=\alpha_j+t-(\beta_i-x+l_{i_2}+\dots+l_{i_{m-1}})\in N$$
	 and this implies 
	 $$x\in N+\beta_i+l_{i_2}+\dots+l_{i_{m-1}}-\alpha_j-t,$$
	 and this is a null set. Note also that, for a fixed $t$, the length of the path $\omega$ is bounded, and therefore there are finitely many such paths, and it follows that $\mathcal N$ is indeed null.

	 Finally, taking the limit in the formulas for $U(t)f_k$ we obtain the formula for $U(t)f$ for any $x$ outside $\mathcal N$.
	
\end{proof}

\begin{corollary}
	\label{cor4.5}
	Let $t>0$ and $f\in \mathscr D_B$. 
	If $x\in\Omega$, $x\in (\alpha_i,\beta_i)$ and $0<t-(\beta_i-x)<\lmin=\min_k l_k$ then 
	$$[U(t)f](x)=\sum_{j=1}^n b_{i,j}f(\alpha_j+t-(\beta_i-x)).$$
	If $f$ is in $L^2(\Omega)$, then this formula holds for almost every such point $x$.
\end{corollary}

\begin{proof}
	Since $0<t-(\beta_i-x)<\lmin$, if $\omega=i_1\dots i_m$ is admissible for $x$ and $t$, then $i_1=i$ and $(\beta_i-x)+l_{i_2}+\dots+l_{i_{m-1}}\leq t<(\beta_i-x)+\lmin$. This means that $m=2$, and so $\omega=ij$, $r(x,t,\omega)=t-(\beta_i-x)$, and $b_\omega=b_{i,j}$. 
\end{proof}

\begin{corollary}
	\label{cor4.6}
	Suppose all the intervals of $\Omega$ have the same length ${\ell}$. If $p\geq 1$ is an integer, then the following properties must hold:
	\begin{enumerate}
		\item If $x\in (\alpha_i,\beta_i)$ and $t>0$ such that $(p-1){\ell}\leq t-(\beta_i-x)< p\ell$, then $r(x,t,\omega)=t-(\beta_i-x)-(p-1){\ell}=:r(x,t)$ for every $\omega\in \mathcal P_{x,t}$. 
		\item For every $x\in (\alpha_i,\beta_i)$ and $t>0$ such that $(p-1){\ell}<t-(\beta_i-x)< p{\ell}$ we have
		$$[U(t)f](x)=\sum_{j=1}^n [B^p]_{i,j}f(\alpha_j+r(x,t)).$$

	\end{enumerate}
\end{corollary}

\begin{proof}
	
	(i) If $\omega=i_1\dots i_m\in \mathcal P_{x,t}$ then 
	$$(\beta_i-x)+l_{i_2}+\dots+l_{i_{m-1}}=(\beta_i-x)+(m-2){\ell}\leq t<(\beta_i-x)+l_{i_2}+\dots+l_{i_m}=(\beta_i-x)+(m-1){\ell},$$
	and therefore $r(x,t,\omega)=t-(\beta_i-x)-(m-2){\ell}$.
	Since, by assumption we have $(p-1){\ell}\leq t-(\beta_i-x)<p{\ell}$, we know that we must have $m-1=p$, so $m=p+1$. This proves (i). 
	
	(ii) Suppose $x\in (\alpha_i,\beta_i)$ and $(p-1){\ell}\leq t-(\beta_i-x)<p{\ell}$. Using the formula in Theorem \ref{thu2}, we write 
	$$[U(t)f](x)=\sum_{\omega\in\mathcal P_{x,t}}b_\omega f(\alpha_{i_m}+r(x,t))$$
	$$=\sum_{j=1}^n\left(\sum_{1\leq i_2,\dots,i_{m-1}\leq n}b_{i,i_2}b_{i_2,i_3}\dots b_{i_{m-2},i_{m-1}}b_{i_{m-1},j}\right)f(\alpha_j+r(x,t))$$
	$$=\sum_{j=1}^n[B^p]_{i,j}f(\alpha_j+r(x,t)),$$
	as desired. 
\end{proof}

Similar results can be obtained for negative values of $t$. Now the point will go through the left-endpoints $\alpha_{i_k}$ and the weights are given by the adjoint matrix $B^*=(b_{i,j}^*)=(\cj b_{j,i})$.

\begin{definition}
	\label{def2.4}{\bf [$t$ negative]}
	Let $x\in\Omega$, $x\in(\alpha_i,\beta_i)$ and $t<0$. An {\it admissible path} for $x$ and $t$ is sequence of indices  $\omega=i_1i_2\dots i_m$, $i_k\in\{1,\dots,n\}$ for all $1\leq k\leq m$, with the following properties
	\begin{enumerate}
		\item $i_1=i$. 
		\item If $t>\alpha_i-x$ then $m=1$. 
		\item If $t\leq\alpha_i-x$ then $m\geq 2$ and 
		$$(\alpha_i-x)-l_{i_2}-l_{i_3}-\dots-l_{i_{m-1}}\geq t >(\alpha_i-x)-l_{i_2}-l_{i_3}-\dots-l_{i_{m-1}}-l_{i_m}.$$
	\end{enumerate}
	If $t\leq\alpha_i-x$, we denote by 
	$$r(x,t,\omega):=(\alpha_i-x)-l_{i_2}-l_{i_3}-\dots-l_{i_{m-1}}-t,$$
	and note that $0\leq r(x,t,\omega)<l_{i_m}$. We call $r(x,t,\omega)$ the {\it remainder} of the path $\omega$. We call $m$ the {\it length} of $\omega$.

We denote by 
$$b_\omega:=b^*_\omega:=b^*_{i_1,i_2}b^*_{i_2,i_3}\dots b^*_{i_{m-1},i_m}.$$

	We call $\beta_{i_m}-r(x,t,\omega)$ the {\it end} of the path $\omega$ (starting at $x$), and we denote it by $\End(x,t,\omega)$.

Let $\mathcal P_{x,t}$ be the set of all admissible paths for $x$ and $t$, and let $\Ends(x,t)$ be the set of all ends of the admissible paths for $x$ and $t$, 
$$\Ends(x,t)=\{\End(x,t,\omega): \omega\in\mathcal P_{x,t}\}.$$

Let $\mathcal P_{x,t}$ be the set of admissible paths for $x$ and $t$. 
\end{definition}

The next theorem can be obtained by applying Theorem \ref{thu2} to the set $-\Omega$ with boundary matrix $B^*$, as described in Remark \ref{reminu}.

\begin{theorem}\label{th2.5}
	Let $t<0$ and $f\in \mathscr D_B$. Let $x\in\Omega$, $x\in (\alpha_i,\beta_i)$.
	
	If $t>x-\alpha_i$ then 
	$$[U(t)f](x)=f(x+t).$$
	If $t\leq x-\alpha_i$, then 
	\begin{equation}
		\label{equ2.9}
		[U(t)f](x)=\sum_{\omega=i_1\dots i_m\in\mathcal P_{x,t}}b_\omega^* f(\beta_{i_m}-r(x,t,\omega)).
	\end{equation}
	If $f\in L^2(\Omega)$ then the formulas above hold for almost every $x\in\Omega$.
\end{theorem}

\section{Spectral sets}\label{sec3}
In this section we focus on the case when our union of intervals $\Omega$ is spectral. In this case, the unitary group $\{U(t)\}$ has the local translation property so $[U(t)f](x)=f(x+t)$ when both $x$ and $x+t$ are in $\Omega$, even when these two points are in different intervals (Theorem \ref{th7.4}). We couple this property with the formula for $U(t)f$ from Theorem \ref{thu2} to extract some geometric properties of our set $\Omega$.

\begin{theorem}\label{thl1}
	Suppose that the set $\Omega$ with boundary matrix $B$ is spectral.  
	\begin{enumerate}
		\item All gaps $\alpha_j-\beta_i>0$, $i\neq j$ are sums of lengths of intervals of $\Omega$. 
		\item   Let $x\in\Omega$, $x\in (\alpha_i,\beta_i)$, $t\in\br$ and $x+t\in(\alpha_j,\beta_j)$ with $i\neq j$.  The following formulas hold:
		\begin{equation}
			\label{eql1.1}
			\sum\{b_\omega: \omega\in\mathcal P_{x,t} , \End(x,t,\omega)=x+t\}=1.
		\end{equation}
		
		\begin{equation}
			\label{eql1.2}
				\sum\{b_\omega: \omega\in\mathcal P_{x,t} , \End(x,t,\omega)=e\}=0,\mbox{ for all }e\in\Ends(x,t), e\neq x+t. 
		\end{equation}

	\end{enumerate}
\end{theorem}

\begin{proof}
	Take $t$ positive first. 
		Since $\Omega$ is spectral, for a function $f\in\mathscr D_B$, with Theorem \ref{thu2},
	\begin{equation}
		\label{eql1}
			f(x+t)=[U(t)f](x)=\sum_{e\in \Ends(x,t)}\sum_{\substack{\omega\in \mathcal P_{x,t}\\ \End(x,t,\omega)=e}}b_\omega f(e).
	\end{equation}
We prove first that $x+t\in \Ends(x,t)$. If not, take a function $f\in\mathscr D_B$ with support in $\Omega$  such that $f(x+t)=1$ and $f(e)=0$ for all $e\in \Ends(x,t)$, and such that $f$ is zero at all the endpoints of the intervals, which means the boundary conditions $Bf(\vec\alpha)=f(\vec\beta)$ are satisfied. Then the left-hand side of \eqref{eql1} is 1 and the right-hand side is zero, a contradiction. 

Thus $x+t\in \Ends(x,t)$. This means that there exists an admissible path $\omega=i_1\dots i_m\in\mathcal P_{x,t}$ with $x+t=\alpha_{i_m}+r(x,t,\omega)$ and $i_m=j$ (since $x+t\in (\alpha_j,\beta_j)$). Then 
$$t=(\beta_i-x)+l_{i_2}+\dots+l_{i_{m-1}}+r(x,t,\omega)=(\beta_i-x)+l_{i_2}+\dots+l_{i_{m-1}}+(x+t-\alpha_{j}).$$
Simplifying:
$$\alpha_j-\beta_i=l_{i_2}+\dots+l_{i_{m-1}},$$
and (i) follows.

Take now in \eqref{eql1} a function $f\in\mathscr D_B$ which is zero at the endpoints of the intervals, and $f(x+t)=1$ and $f(e)=0$ for all $e\in\Ends(x,t)$, $e\neq x+t$. Plug this in \eqref{eql1}, we obtain \eqref{eql1.1}. 

Now take an endpoint $e_0\in \Ends(x,t)$ , $e_0\neq x+t$. Take a continuous function $f\in\mathscr D_B$, with $f(e_0)=1$, $f(x+t)=0$ and $f(e)=0$ for all $e\in \Ends(x,t)$, $e\neq e_0$. If $e_0$ is not an endpoint of an interval, then one can make $f$ zero at all the endpoints of the intervals, to make sure the boundary conditions $Bf(\vec\alpha)=f(\vec\beta)$ are satisfied. If $e_0$ is an endpoint, note that by our conventions in Definition \ref{defu1}, for $t>0$, all points in $\Ends(x,t)$ cannot be the right-endpoints $\beta_k$. Therefore $e_0$ has to be one of the $\alpha$'s,  $e_0=\alpha_{i_0}$, and any other point in $\Ends(x,t)$, if it is also an endpoint of an interval, then it has to be one of the $\alpha$'s. Thus, we can construct a function $f\in\mathscr D_B$, which is 1 at $e_0$, 0 at all other points in $\Ends(x,t)$, $f(x+t)=0$, and we can make sure that the values of $f$ at the $\beta$'s are correct so that the boundary condition $Bf(\vec\alpha)=f(\vec \beta)$ are satisfied, which amounts to $f(\vec\beta)=B\delta_{i_0}$ (just use a piecewise linear continuous function, or some smooth interpolation).

 Finally, if we plug $f$ in \eqref{eql1}, we obtain \eqref{eql1.2}. 

For the case when $t$ is negative, one can apply the positive case to the set $-\Omega$ with boundary matrix $B^*$ and $-t$, as explained in Remark \ref{reminu}, and then use Proposition \ref{prminu}.
\end{proof}

\begin{remark}
An alternative proof of Theorem \ref{thl1}(i) was independently obtained by Kolountzakis, Lev, and Matolcsi in \cite[Theorem 3.1]{KLM25} around the time this paper was written.
\end{remark}

\begin{proposition}
	\label{prg}
	Suppose $\Omega$ is spectral. Assume two of the intervals of $\Omega$ have a common endpoint, $\alpha_{i+1}=\beta_i$. Then $b_{i,i+1}=1$, and $b_{i,j}=0$ for $j\neq i+1$, $b_{j,i+1}=0$ for $j\neq i$. 
\end{proposition}

\begin{proof}
	Let $\epsilon>0$ with $2\epsilon<\lmin:=\min_k l_k$. Let $x=\beta_i-\epsilon\in(\alpha_i,\beta_i)$ and let $t=2\epsilon$, so that $x+t=\alpha_{i+1}+\epsilon\in (\alpha_{i+1},\beta_{i+1})$. Let $\omega_0=i(i+1)$. $\omega_0$ is an admissible path for $x$ and $t$ with $\End(x,t,\omega)=\alpha_{i+1}+\epsilon=x+t$. 
	
	If $\omega=i_1\dots i_m$ is an admissible path in $\mathcal P_{x,t}$ with $\End(x,t,\omega)=x+t=\alpha_{i+1}+\epsilon$, then $i_1=i$, $i_m=i+1$, and 
	$$2\epsilon=t=\epsilon+l_{i_2}+\dots+l_{i_{m-1}}+\epsilon.$$
	This implies that $m=2$ and so $\omega=i(i+1)$. From \eqref{eql1.1}, we get that $b_{i,i+1}=b_{\omega_0}=1$. Since $B$ is unitary, the $i$-th row and the $i+1$-st column have norm one, which implies that all the other entries in this row and this column are zero.
\end{proof}

\begin{proposition}
	\label{prl3}
	Suppose $\Omega$ is spectral. Suppose one of the gaps, $\alpha_{i+1}-\beta_i$ is equal to the minimal length $\lmin=\min_k l_k$. 
	\begin{enumerate}
		\item The following relation holds:
		\begin{equation}
			\label{eql3.1}
			\sum_{j: l_j=\lmin}b_{i,j}b_{j,i+1}=1.
		\end{equation}
		\item $b_{i,j}=0$ and $b_{j,i+1}=0$, for all $j$ with $l_j>\lmin$
	\end{enumerate}
\end{proposition}

\begin{proof}
	Let $\epsilon>0$ and $2\epsilon<\lmin$. Let $x=\beta_i-\epsilon\in (\alpha_i,\beta_i)$, $t=\lmin+2\epsilon$ so that $x+t=\alpha_{i+1}+\epsilon\in(\alpha_{i+1},\beta_{i+1})$. 
	This shows that, if $l_j=\lmin$ then the path $\omega_j=ij(i+1)$ is admissible with $\End(x,t,\omega_j)=x+t$. 
	
	If $\omega=i_1\dots i_m$ is an admissible path $\omega\in \mathcal P_{x,t}$ with $\End(x,t,\omega)=x+t$, then we must have $i_1=i$, $i_m=i+1$, $r(x,t,\omega)=\epsilon$, and 
	$$\lmin+2\epsilon=t=\epsilon+l_{i_2}+\dots+l_{i_{m-1}}+\epsilon.$$
	This means that $m=3$, and $l_{i_2}=\lmin$, and so $\omega=\omega_{i_2}$. We have $b_{\omega_j}=b_{i,j}b_{j,i+1}$. From \eqref{eql1.1}, we obtain \eqref{eql3.1}.

	For (ii), take $\epsilon$ small enough so that $\lmin +2\epsilon< l_j$ for all $j$ with $l_j>\lmin$. For an index $j$ with $l_j>\lmin$, we can write 
	$$\lmin+2\epsilon=t=\epsilon+(\lmin+\epsilon),$$ 
	and we can consider $\lmin+\epsilon$ as a remainder inside the interval $(\alpha_j,\beta_j)$; thus $\omega_0=ij$ is an admissible path $\omega_0\in\mathcal P_{x,t}$ and $\End(x,t,\omega_0)=\alpha_j+(\lmin+\epsilon)$.
	
	If $\omega=i_1\dots i_m$ is an admissible path with $\End(x,t,\omega)=\End(x,t,\omega_0)=\alpha_j+(\lmin+\epsilon)$, then $i_1=i$, $i_m=j$, $r(x,t,\omega)=\lmin+\epsilon$ and 
	$$\lmin+2\epsilon=t=\epsilon+l_{i_2}+\dots+l_{i_{m-1}}+(\lmin+\epsilon).$$
	This implies that $m=2$, and so $\omega=ij=\omega_0$. With \eqref{eql1.2}, we obtain that $b_{i,j}=b_{\omega_0}=0$. 
	
	Take now, $x=\alpha_{i+1}+\epsilon$, $t=\lmin+2\epsilon$, so that $x-t=\beta_i-\epsilon\in(\alpha_i,\beta_i)$. For an index $j$ with $l_j>\lmin$, we can write 
	$$-t=-\epsilon-(\lmin+\epsilon),$$ 
	and we can consider $\lmin+\epsilon$ as a remainder in the interval $(\alpha_j,\beta_j)$ for a path for $x$ and $-t$. Therefore the path $\omega_0=(i+1)j$ is admissible $\omega_0\in \mathcal P_{x,-t}$ and $\End(x,t,\omega_0)=\beta_j-(\lmin+\epsilon)$. 
	
	If $\omega=i_1\dots i_m\in\mathcal P_{x,-t}$ with $\End(x,t,\omega)=\beta_j-(\lmin+\epsilon)$, then $i_1=i+1$, $i_m=j$, $r(x,-t,\omega)=\lmin+\epsilon$ and 
	$$-t=-\epsilon-l_{i_2}-\dots-l_{i_{m-1}}-(\lmin+\epsilon).$$
	This forces $m=2$ and therefore $\omega=(i+1)j=\omega_0$. Using again \eqref{eql1.2}, we obtain $b_{i+1,j}^*=b_{\omega_0}^*=0$ so $b_{j,i+1}=0$.

\end{proof}

\begin{proposition}\label{prd}
	Suppose $\Omega$ is spectral and has at least two intervals. Then $|b_{k,k}|\neq 1$ for all $1\leq k\leq n$. 
\end{proposition}

\begin{proof}
	Suppose $|b_{i,i}|=1$ for some $1\leq i\leq n$. Then, since each row and column of $B$ has norm one, it follows that $b_{i,k}=0$ for all $1\leq k\leq n$ with $k\neq i$. 
	
	Take now $j\neq i$, $x\in (\alpha_i,\beta_i)$ and $t\in\br$ such that $x+t\in (\alpha_j,\beta_j)$. With \eqref{eql1.1}, 
	$$\sum\left\{b_\omega : \omega\in\mathcal P_{x,t}, \End(x,t,\omega)=x+t\right\}=1.$$

	However, if $\omega=i_1\dots i_m\in \mathcal P_{x,t}$ and $\End(x,t,\omega)=x+t$ then $i_1=i$, $i_m=j$. If, in addition $b_\omega\neq 0$, then $b_{i_1,i_2}\neq 0$ so $i_2$ has to be $i$. By induction, $i_3=\dots= i_m=i$ so $i=j$, a contradiction. 
	\end{proof}

	\begin{proposition}
		\label{prb1}
		Suppose $\Omega$ is spectral and $b_{i,j}=e^{2\pi i\theta_0}$ for some $i,j\in\{1,\dots,n\}$ and $\theta_0\in\br$. 
		\begin{enumerate}
			\item The spectrum $\Lambda$ of $\Omega$ is contained in $\frac{1}{\alpha_j-\beta_i}(\bz-\theta_0)$. 
			\item The sets $\{\Omega+k(\alpha_j-\beta_i)\bz: k\in\bz\}$ are disjoint. 
			\item Let $E$ be a subset of $\Omega$ of positive measure and $k\in\bz$. Let 
		$$\tilde \Omega=(\Omega\setminus E)\cup (E+k(\alpha_j-\beta_i)).$$	
		Then $\tilde \Omega$ is spectral with the same spectrum $\Lambda$.
		\end{enumerate}
	\end{proposition}
	
	\begin{proof}
		By Theorem \ref{th7.3}, if $\lambda\in\Lambda$, then 
		\begin{equation}\label{eqb1.1}
			Be_\lambda(\vec\alpha)= e_\lambda(\vec\beta).
		\end{equation}
		
		 Since $|b_{i,j}|=1$ and $B$ is unitary, we have $b_{i,k}=0$ for all $k\neq j$. Then, considering the $i$-th component in \eqref{eqb1.1}: $b_{i,j}e_\lambda(\alpha_j)=e_\lambda(\beta_i)$, which means that $\theta_0+\lambda(\alpha_j-\beta_i)\in\bz$ and (i) follows. 
		
For (ii) we use a lemma:

\begin{lemma}
	\label{lemdis}
	If a set $\Omega$ is spectral with a spectrum $\Lambda$ contained in $\frac{1}{a}(\bz-\theta_0)$ for some $a\in\br$, $a\neq 0$, and $\theta_0\in\br$, then the sets 
	$\{\Omega+ka : k\in\bz\}$ are disjoint . 
\end{lemma}

\begin{proof}
	Suppose, for contradiction, that for some $k\in\bz$, $k\neq 0$, the set $E=\Omega\cap (\Omega+ka)$ has positive Lebesgue measure. Then $E\subset\Omega$ and $E-ak\subset\Omega$. Define the function $f=e^{-2\pi ik\theta_0}\chi_E-\chi_{E-ak}\in L^2(\Omega)$. For any $\lambda\in\Lambda$, we have 
	$$\ip{f}{e_\lambda}_{L^2(\Omega)}=\int_\Omega f(x)e^{-2\pi i\lambda x}\,dx=\int_E e^{-2\pi i(k\theta_0+\lambda x)}\,dx-\int_{E-ak}e^{-2\pi i\lambda x}\,dx$$
	$$=\int_Ee^{-2\pi i(k\theta_0+\lambda x)}\,dx -\int_{E-ak} e^{-2\pi i(k\theta_0+\lambda(x+ak))}\,dx\, (\mbox{since }(\lambda a +\theta_0) k\in\bz))$$
	$$=\int_Ee^{-2\pi i(k\theta_0+\lambda x)}\,dx-\int_Ee^{-2\pi i(k\theta_0+\lambda y)}\,dy=0.\, (\mbox{with the substitution }y=x+ak).$$

	Thus $\ip{f}{e_\lambda}=0$ for all $\lambda$ in the spectrum $\Lambda$; but this implies that $f=0$, a contradiction. 
	
\end{proof}

With (i) and Lemma \ref{lemdis}, we obtain (ii). 

For (iii), we use another lemma. 

\begin{lemma}
	\label{lemmov}
	Suppose $\Omega$ is a spectral set with spectrum $\Lambda$ contained in $\frac1a(\bz-\theta_0)$, $a\in \br$, $a\neq 0$, $\theta_0\in\br$. Let $E$ be a subset of $\Omega$ with positive measure. Let 
	$$\tilde \Omega=(\Omega\setminus E)\cup(E+ka).$$
	Then $\tilde \Omega$ has the same spectrum $\Lambda$. 
\end{lemma}

\begin{proof}
	Of course, we can assume that $k\neq 0$. 
	First of all, from Lemma \ref{lemdis}, we know that, $E+ka$ is disjoint from $\Omega$. 
	
	Define the operator $\mathcal T:L^2(\Omega)\to L^2(\tilde\Omega)$, 
	\begin{equation}
		\label{eqmov}
		\mathcal Tf(x)=\left\{
		\begin{array}{cc}
			f(x),& x\in\Omega\setminus E,\\
			e^{-2\pi i k\theta_0}f(x-ka),&x\in E+ka.
			\end{array}\right.
	\end{equation}
	
	It is clear that the operator is unitary. Also, for all $\lambda\in\Lambda$, $\mathcal T e_\lambda=e_\lambda$, since $k(\lambda a+\theta_0)\in\bz$. Thus $\Lambda$ is a spectrum for $\tilde \Omega$. 
\end{proof}

With Lemma \ref{lemmov}, the proof of Proposition \ref{prb1} is complete. 
	\end{proof}
\begin{corollary}
	\label{cormov}
		Suppose $\Omega$ is spectral and $b_{i,j}=e^{2\pi i\theta_0}$ for some $i,j\in\{1,\dots,n\}$ and $\theta_0\in\br$. Let $\tilde\Omega$ be the union of intervals obtained from $\Omega$ by moving the interval $(\alpha_j,\beta_j)$ at the end of the interval $(\alpha_i,\beta_i)$; so 
		$$\tilde \Omega=\bigcup_{k\neq i,j}(\alpha_k,\beta_k)\bigcup (\alpha_i,\alpha_i+l_i+l_j).$$
		Then $\tilde \Omega$ has the same spectrum as $\Omega$. 
		Similarly, when the interval $(\alpha_i,\beta_i)$ is moved at the beginning of the interval $(\alpha_j,\beta_j)$.
	
		\end{corollary}

\begin{proof}
This follows from Proposition \ref{prb1}, with $E=(\alpha_j,\beta_j)$.
\end{proof}

	\section{Multiplicative unitary groups}\label{sec4}
	In this section we will impose some extra restrictions on the unitary group $\{U(t)\}$ and see how this affects the geometry of $\Omega$ and the boundary matrix $B$. The first one, the multiplicative condition, was proposed by Segal to Fuglede, in order to obtain a more precise form of Fuglede's result Theorem \ref{th2.1}. Fuglede proved in his seminal paper \cite{Fug74} that (in any dimension) if a connected set $\Omega$ is spectral and the associated unitary group satisfies the multiplicative condition, then $\Omega$ is a fundamental domain for a lattice and the spectrum of $\Omega$ is the dual lattice. We will prove that if our union of intervals is spectral and the unitary group is multiplicative, then $\Omega$ tiles with $L\bz$, where $L=\mathfrak m(\Omega)$, $\Omega$ has spectrum $\frac1L\bz$, and we obtain detailed information about the matrix $B$ and the geometry of $\Omega$.  
	\begin{definition}
		\label{defm1}
		For $t\in\br$, we say that $U(t)$ is {\it multiplicative} if for all $f,g\in L^2(\Omega)$ with $fg\in L^2(\Omega)$, 
		$$U(t)(fg)=U(t)f\cdot U(t) g.$$
		We also define the set 
		$$\mathscr T_B=\{t\in\br : U(t)\mbox{ is multiplicative}\}.$$ 
		If $\mathscr T_B=\br$ then we say that the unitary group $\{U(t)\}$ is {\it multiplicative}. 
	\end{definition}
	
	\begin{lemma}
		\label{lemm2}
		Let $t\in\br$. $U(t)$ is multiplicative if and only if for any $f,g\in L^\infty(\Omega)$, $U(t)fg=U(t)f\cdot U(t)g$.
	\end{lemma}
	
	\begin{proof}
	The direct implication is clear. For the converse, let $f,g\in L^2(\Omega)$ with $fg\in L^2(\Omega)$. Let 
	$$f_n=f\chi_{\{x\in \Omega: |f(x)|\leq n\}},\quad g_n=g\chi_{\{x\in \Omega: |g(x)|\leq n\}},\quad(n\in\bn).$$
	Then $\lim f_n=f$ pointwise and $|f_n-f|^2\leq |f|^2$, therefore, by the Lebesgue Dominated Convergence Theorem, $\lim f_n=f$ in $L^2(\Omega)$. The same is true for $\{g_n\}$. 
	
 Also, $\lim f_ng_n=fg$ pointwise and $|f_ng_n-fg|^2\leq |2fg|^2$, and with the Lebesgue Dominated Convergence Theorem, $\lim f_ng_n=fg$ in $L^2(\Omega)$. 
 
 Then, $U(t)(f_ng_n)$ converges to $U(t)(fg)$ in $L^2(\Omega)$. On the other hand, $U(t)f_n$ converges to $U(t)f$ in $L^2(\Omega)$ and similarly for $U(t)g_n$. Therefore, replacing the sequence by a subsequence, we get that $\lim U(t)f_n=U(t)f$ and $\lim U(t)g_n=U(t)g$ pointwise, and so 
 $$\lim U(t)(f_ng_n)=\lim U(t)f_n\cdot U(t)g_n=U(t)f\cdot U(t)g\mbox{ pointwise a.e.}.$$
 Thus $U(t)fg=U(t)f\cdot U(t)g.$
	
	\end{proof}
	\begin{proposition}
		\label{prm2}
		The set $\mathscr T_B$ is a closed subgroup of $\br$, so $\mathscr T_B=\br$ or $\mathscr{T}_B=p\bz$ for some $p\geq 0$. 
	\end{proposition}
	
	\begin{proof}
$\mathscr T_B$ is nonempty because $U(0)=I_{L^2(\Omega)}$ so $0\in\mathscr T_B$.		We use Lemma \ref{lemm2}.
		Let $t,s\in\mathscr T_B$. Thus, for all $f,g\in L^\infty(\Omega)$, we have 
		$$U(s+t)(fg)=U(s)[U(t)(fg)]=U(s)[U(t)f\cdot U(t) g]=U(s)U(t)f\cdot U(s)U(t)g=U(s+t)f\cdot U(s+t)g.$$
		This means that $t+s\in\mathscr T_B$. Note that, by Theorem \ref{thu2}, if $h\in L^\infty(\Omega)$ then $U(s)h\in L^\infty(\Omega)$, for any $s\in\br$. 
	
	If $f,g\in L^\infty(\Omega)$, then $U(-t)f,U(-t)g\in L^\infty(\Omega)$. Then 
	$$U(t)[U(-t)f\cdot U(-t)g]=U(t)U(-t)f\cdot U(t)U(-t)g=fg,$$
	therefore $U(-t)f\cdot U(-t)g=U(-t)(fg)$. This shows that $-t\in\mathscr T_B$ and $\mathscr T_B$ is a group. 
	
	To see that $\mathscr T_B$ is closed, let $\{t_n\}$ be a sequence in $\mathscr T_B$ with $t_n\to t$ for some $t\in\br$. Since $\{U(t)\}$ is strongly continuous, we know that, for $f,g\in L^\infty(\Omega)$, 
	$U(t)(fg)=\lim U(t_n)(fg)$ in $L^2(\Omega)$. Also $\lim U(t_n)(fg)=\lim U(t_n)f\cdot U(t_n)g=U(t)f\cdot U(t)g$ in $L^2(\Omega)$. Therefore $U(t)(fg)=U(t)f\cdot U(t)g$ and $t\in \mathscr T_B$. 
	
	\end{proof}

	\begin{definition}
		\label{defm3}
		
		For a permutation $\sigma\in S_n$ -- the symmetric group on $n$ symbols --  we denote by $P^\sigma$ the permutation matrix obtained from the identity matrix $I_n$ by permuting the columns according to the permutation $\sigma$, 
		$$P^\sigma_{i,j}=\left\{\begin{array}{cc}
			1&,\mbox{ if }\sigma(i)=j,\\
			0&,\mbox{ otherwise.}
		\end{array}\right.$$
	\end{definition}

	\begin{theorem}\label{thm4}
		The set $\mathscr T_B=\br$ if and only if $B$ is a permutation matrix. 
	\end{theorem}
	
	\begin{proof}
		Suppose that $B=P^\sigma$ for some permutation $\sigma\in S_n$. For $f\in L^\infty(\Omega)$, and for almost every $x\in (\alpha_i,\beta_i)$ and $0<t<\lmin=\min_k l_k$, by Corollary \ref{cor4.5}, we have that 
		$$[U(t)f](x)=\left\{\begin{array}{cc}
		f(x+t),&\mbox{ if }(\beta_i-x)>t,\\
		f(\alpha_{\sigma(i)}+t-(\beta_i-x)),&\mbox{ if } (\beta_i-x)<t,
		\end{array}\right.$$
	Then, for all $f,g\in L^\infty(\Omega)$, and for a.e. $x\in(\alpha_i,\beta_i)$, 
	
		$$[U(t)](fg)=\left\{\begin{array}{cc}
			f(x+t)g(x+t),&\mbox{ if }(\beta_i-x)>t,\\
			f(\alpha_{\sigma(i)}+t-(\beta_i-x))g(\alpha_{\sigma(i)}+t-(\beta_i-x)),&\mbox{ if } (\beta_i-x)<t,
		\end{array}\right.$$
		$$=U(t)f\cdot U(t)g.$$
		
		Thus $U(t)$ is multiplicative for $0<t<\lmin$. With Proposition \ref{prm2}, we get that $\mathscr T_B=\br$.

		Conversely, assume that $\mathscr T_B=\br$. Let $0<t<\lmin$ and $f,g\in L^\infty(\Omega)$. Then for a.e.  $x\in (\alpha_i,\beta_i)$ with $t-(\beta_i-x)<\lmin$, by Corollary \ref{cor4.5}, we have
		$$[U(t)(fg)](x)=\sum_{j=1}^n b_{i,j}f(\alpha_j+t-(\beta_i-x))g(\alpha_j+t-(\beta_i-x)).$$
 Since $U(t)$ is multiplicative we must also have that 
		$$[U(t)(fg)](x)=[U(t)f](x)\cdot [U(t)g](x)$$
		$$=\sum_{j=1}^n b_{i,j}f(\alpha_j+t-(\beta_i-x))\cdot \sum_{j=1}^nb_{i,j} g(\alpha_j+t-(\beta_i-x))$$
		$$=\sum_{j,j'=1}^n b_{i,j}b_{i,j'}f(\alpha_j+t-(\beta_i-x))g(\alpha_{j'}+t-(\beta_i-x)).$$
		For each $1\leq j_0\leq n$ we can choose $f_{j_0}\in L^\infty(\Omega)$ such that $f_{j_0}(\alpha_j+t-(\beta_i-x))=\delta_{j,j_0}$ for $x\in (\alpha_i,\beta_i)$ with $t-(\beta_i-x)<\lmin$ for all $1\leq j\leq n$. Using $f=f_{j_0}$, $g=f_{j_0'}$, for some $j_0\neq j_0'$ in $\{1,\dots,n\}$, in the two formulas above, we obtain 
		$$b_{i,j_0}b_{i,j_0'}=[U(t)(f_{j_0}f_{j_0'})](x)=0.$$
		This means that there can be at most one index $1\leq j_0\leq n$ with $b_{i,j}\neq 0$. If we use $f=g=f_{j_0}$ in the two formulas above we obtain 
		$$b_{i,j_0}=[U(t)(f_{j_0}^2)](x)=(b_{i,j_0})^2,$$ 
		implying that $b_{i,j_0}=1$. We denote by $\sigma(i):=j_0$. 
		
		Suppose  now that for $i\neq i'$ we have $\sigma(i)=\sigma(i')=j$. This means that $b_{i,j}=b_{i',j}=1$. This is impossible since $B$ is unitary and so the $j$-th column has norm 1. Thus $\sigma$ is a permutation and $B=P^\sigma$. 
	\end{proof}
	
	\begin{definition}
		\label{defm4}
		Let $L>0$. Two measurable subsets $A$ and $B$ are said to be {\it translation congruent modulo} $L\bz$, if there exist measurable partitions $\{A_i\}_{i\in I}$ of $A$ and $\{B_i\}_{i\in I}$ of $B$, with $I$ at most countable, and integers $k_i\in \bz$ such that $A_i+k_i=B_i$ up to measure zero. 
	\end{definition}
	
	\begin{theorem}
		\label{thm5}
		Suppose $\Omega$ is spectral and $\{U(t)\}$ is multiplicative. 
		\begin{enumerate}
			\item The matrix $B=P^\sigma$ where $P$ is a cycle of length $n$. 
			\item Let $L:=\mathfrak m(\Omega)=l_1+\dots+l_n$. The spectrum $\Lambda=\frac{1}{L}\bz$ and $\Omega$ tiles with tiling set $L\bz$. 
			\item Translate the intervals of $\Omega$ as follows: move the interval $(\alpha_{\sigma(1)},\beta_{\sigma(1)})$ at the end of $(\alpha_1,\beta_1)$, that is 
			$$(\alpha_{\sigma(1)},\beta_{\sigma(1)})+(\beta_1-\alpha_{\sigma(1)})=(\beta_1,\alpha_1+l_1);$$
			move the interval  $(\alpha_{\sigma^2(1)},\beta_{\sigma^2(1)})$ at the end of the previous one, that is 
			$$(\alpha_{\sigma^2(1)},\beta_{\sigma^2(1)})+(\beta_1-\alpha_{\sigma(1)})+(\beta_{\sigma(1)}-\alpha_{\sigma^2(1)})=(\alpha_1+l_1,\alpha_1+l_1+l_2);$$
			and so on, at the end move the interval  $(\alpha_{\sigma^{n-1}(1)},\beta_{\sigma^{n-1}(1)})$ at the end of the previous one, that is 
				$$(\alpha_{\sigma^{n-1}(1)},\beta_{\sigma^{n-1}(1)})+(\beta_1-\alpha_{\sigma(1)})+(\beta_{\sigma(1)}-\alpha_{\sigma^2(1)})+\dots+(\beta_{\sigma^{n-2}(1)}-\alpha_{\sigma^{n-1}(1)})$$$$=(\alpha_1+l_1+l_2+\dots+l_{n-1},\alpha_1+l_1+l_2+\dots+l_n).$$
The union is $(\alpha_1,\alpha_1+L)$ (except for finitely many endpoints). All the translations in the previous moves are in $L\bz$. They show how $\Omega$ is translation congruent to $(\alpha_1,\alpha_1+L)$ modulo $L\bz$. 
		\end{enumerate}
	\end{theorem}
	
	\begin{proof}
		We know from Theorem \ref{thm4} that $B=P^\sigma$ for some permutation $\sigma$. We just have to show that $\sigma$ is a cycle of length $n$. 
		
	Consider $i<j$ in $\{1,\dots, n\}$. Let $x\in (\alpha_i,\beta_i)$ and $t>0$ so that $x+t\in (\alpha_j,\beta_j)$. With Theorem \ref{thl1}, we know that 
	$$\sum\{b_\omega : \omega\in \mathcal P_{x,t}, \End(x,t,\omega)=x+t\}=1.$$
	Thus, there is a path $\omega=i_1\dots i_m$ with $i_1=i$, $i_m=j$ and $b_\omega=b_{i_1,i_2}\dots b_{i_{m-1},i_m}\neq 0$. But then $i_{k}=\sigma(i_{k-1})$ for all $2\leq k\leq m$, and therefore $j=\sigma^{m-1}(i)$. Thus all the points belong to the same cycle of $\sigma$, hence $\sigma$ is a cycle of length $n$. This proves (i).  
	
	Consider now 
	$$\Lambda^*=\{t\in\br: t\lambda\in\bz \mbox{ for all }\lambda\in\Lambda\}.$$
	Clearly, $\Lambda^*$ is a closed subgroup of $\br$.
	
	Note that for any $1\leq i\leq n$, since $b_{i,\sigma(i)}=1$, with Proposition \ref{prb1}(i), we have that $\Lambda$ is contained in $\frac{1}{\alpha_{\sigma(i)}-\beta_{i}}\bz$. This implies that $\alpha_{\sigma(i)}-\beta_i$ is in $\Lambda^*$. This shows that all the translations that appear in (iii) are in $\Lambda^*$. 
	
	Since $\Lambda^*$ is a nontrivial closed subgroup of $\br$, $\Lambda^*=a\bz$ for some $a>0$. Then $\lambda a\in\bz$ for all $\lambda\in\Lambda$ so $\Lambda$ is a subset of $\frac1a\bz$. Using Lemma \ref{lemmov} at each step in (iii), we obtain that after each move, we have a set with the same spectrum $\Lambda$ as $\Omega$, and therefore the final set which is the interval $(\alpha_1,\alpha_1+L)$ has the same spectrum $\Lambda$. 
	
	The spectra of an interval of length $L$ are $\frac1L\bz+\lambda_0$ for some $\lambda_0\in\br$. However, since 
	$P^\sigma(e_0(\vec\alpha))=e_0(\vec \beta)$, we get with Theorem \ref{th7.3} that $0\in\Lambda$. Thus $\Lambda=\frac1L\bz$. From the previous argument, all the translations in (iii) are in $\Lambda^*=L\bz$ and so $\Omega$ is translation congruent to $(\alpha_1,\alpha_1+L)$ modulo $L\bz$, which implies that $\Omega$ tiles by $L\bz$, a fact that can also be obtained from Fuglede's original paper \cite{Fug74}: a set has spectrum a lattice if and only if it tiles with the dual lattice. 
	\end{proof}

	\begin{theorem}
		\label{thm6}
		Suppose all the intervals of $\Omega$ have the same length $l_k={\ell}$ for all $1\leq k\leq n$ and that $\mathscr T_B\neq\{0\}$. If $t_0\in\mathscr T_B$, $t_0>0$  and $p \geq 1$ is an integer such that $(p-1)\ell< t_0\leq p\ell$, then $B^p$ is a permutation matrix.  
	\end{theorem}
	
	\begin{proof}
We replicate the proof of Theorem \ref{thm4} using Corollary \ref{cor4.6} instead. Let $f,g\in L^\infty(\Omega)$. 
For a.e. $x\in (\alpha_i,\beta_i)$ such that $(p-1){\ell}<t_0-(\beta_i-x)<p{\ell}$, using Corollary \ref{cor4.6}, and setting $r(x)=t-(\beta_i-x)-(p-1){\ell}$, we have
$$[U(t_0)(fg)](x)=\sum_{j=1}^n[B^p]_{i,j}f(\alpha_j+r(x))g(\alpha_j+r(x)).$$

Since $U(t_0)$ is multiplicative, we must also have that 
$$[U(t_0)(fg)](x)=[U(t_0)f](x)\cdot [U(t_0)g](x)=\sum_{j=1}^n[B^p]_{i,j}f(\alpha_j+r(x))\cdot\sum_{j=1}[B^p]_{i,j} g(\alpha_j+r(x)).$$
For each $1\leq j_0\leq n$, chose $f\in L^\infty(\Omega)$ such that for a.e. $x\in (\alpha_i,\beta_i)$ with $(p-1)\ell<t_0-(\beta_i-x)<p\ell$, we have $f(\alpha_{j}+r(x))=\delta_{j,j_0}$ for all $1\leq j\leq n$. Use $f=f_{j_0}$, $g=f_{j_0'}$ in the two formulas above and we get
$$[B^p]_{i,j_0}[B^p]_{i,j_0'}=[U(t)(f_{j_0}f_{j_0'})](x)=0,$$ 
for all $j_0\neq j_0'$. Hence, there can be at most one index $1\leq j_0:=\sigma(i)\leq n$ such that $[B^p]_{i,j_0}\neq 0$. Since $B$ is unitary, so is $B^p$, and therefore, there is exactly one index such that $|[B^p]_{i,\sigma(i)}|=1$ and $[B^p]_{i,j}=0$ for all $j\neq \sigma(i)$. Next, observe that if we take $j_0=j_0'=\sigma(i)$, we get 
$$[B^p]_{i,\sigma(i)}=[U(t_0)(f_{j_0}^2)](x)=[B^p]_{i,\sigma(i)}^2,$$
implying that $[B^p]_{i,\sigma(i)}=1$. If $\sigma(i)=\sigma(i')=j$, then $[B^p]_{i,j}=[B^p]_{i',j}=1$, which contradicts the fact that $B^p$ is unitary. Thus $\sigma$ is a permutation and $B$ is a permutation matrix. 

	\end{proof}
	
	\begin{remark}
		\label{rem8}
		Assume that $\Omega$ has equal length intervals, has measure 1, and that the spectrum $\Lambda$ is rational, $\Lambda\subset\mathbb Q$. (It is an open question whether any such spectral set of measure 1 has a rational spectrum). We know from \cite{BM11,Kol12,IK13}, that the spectrum has to be periodic and the period $p$ is a positive integer. This will imply that 
		$$\Lambda=\{\lambda_1,\dots,\lambda_{r}\}+p\bz,$$
		for some rational numbers $\lambda_1,\dots,\lambda_r$. Then all the rational numbers in $\Lambda$ have a common denominator, which we denote by $D$. With Proposition \ref{prm1}(i), we get that the eigenvalues of $B$ have to be of the form $\lambda_B=e^{2\pi i\lambda \ell}$; since $\mathfrak m(\Omega)=1$, we have that $\ell=1/n$. Then $\lambda_B^{Dn}=1$, for all eigenvalues $\lambda_B$ of $B$. With the Spectral Theorem for the unitary matrix $B$, we get that $B^{Dn}=I_n$.
		
		Conversely, in the same situation -- when $\Omega$ is spectral, of measure 1, and with equal length intervals -- if $\mathscr T_B\neq\{0\}$, then with Theorem \ref{thm6}, $B^p$ is a permutation matrix. Then $B^{pk}=I_n$ for some integer $k\geq 0$ ($k$ a common multiple of the lengths of the cycles of the permutation). Then the eigenvalues of $B$ are roots of order $pk$ of unity, and, with Proposition \ref{prm1}(i), we get that $\Lambda$ is rational.

	\end{remark}

In what follows we consider the Forelli condition on $\{U(t)\}$, and we will see that we will obtain similar results under this restriction as for the multiplicative case. The Forelli condition was introduced in this context by Jorgensen in \cite{Jor82}, where it is proved (in any dimension) that a connected spectral set with the Forelli condition has to be a fundamental domain for a lattice. 

\begin{definition}
	\label{defo}
	For $g\in L^\infty(\Omega)$, let $M_f$ be the multiplication operator on $L^2(\Omega)$, $M_gf=gf$, for all $f\in L^2(\Omega)$. Let 
	$$\mathcal M:=\{M_g : g\in L^\infty(\Omega)\}.$$
	For $t\in\br$, we say that $U(t)$ satisfies the {\it Forelli condition} if 
	$$U(t)\mathcal M U(t)^*=\mathcal M.$$
	In this case $U(t)$ induces an automorphism on $\mathcal M$, which we denote by $\gamma_t$: 
	$$U(t)M_fU(t)^*=M_{\gamma_t(f)},\quad(f\in L^\infty(\Omega)).$$

	Let
	$$\mathscr F_B:=\{t\in\br : U(t)\mbox{ satisfies the Forelli condition}\}.$$
	We say that $\{U(t)\}$ satisfies the Forelli condition if $\mathscr F_B=\br$.
\end{definition}

\begin{proposition}
	\label{prf1}
	$\mathscr F_B$ is a closed subgroup of $\br$. 
\end{proposition}

\begin{proof}
	Let $t_1,t_2\in\mathscr F_B$ and $g\in L^\infty(\Omega)$. Then 	
	$$U(t_1+t_2)M_gU(t_1+t_2)^*=U(t_1)U(t_2)M_gU(t_2)^*U(t_1)^*=U(t_1)M_{\gamma_{t_2}(g)}U(t_1)^*=M_{\gamma_{t_1}(\gamma_{t_2}(g))}.$$
	This means that $t_1+t_2\in \mathscr F_B$. 
	
	If $t\in\mathscr F_B$, then $U(t)\mathcal MU(t)^*=\mathcal M$, which implies that $\gamma_t$ is indeed onto and one-to-one, so it has an inverse, and if $g=\gamma_t(f)$, for some $f\in L^\infty(\Omega)$, then 
	$$U(-t)M_{g}U(t)=M_f=M_{\gamma_t^{-1}(g)}.$$
	In particular, $-t\in \mathscr F_B$ and therefore $\mathscr F_B$ is a subgroup of $\br$. 
	
	We prove that $\mathscr F_B$ is closed. Let $t_n\in\mathscr F_B$ with $t_n\to t$. Let $g\in L^\infty(\Omega)$.  Since $\{U(t)\}$ is strongly convergent, 
	$M_{\gamma_{t_n}(g)}=U(t_n)M_gU(t_n)^*$ converges strongly to $U(t)M_gU(t)^*$. Since $\mathcal M$ is a von Neumann algebra, it is closed in the strong operator topology and therefore, $U(t)M_gU(t)^*\in\mathcal M$. This proves $U(t)\mathcal MU(t)^*\subseteq \mathcal M$. Repeating the argument for $U(-t)$, we obtain the reverse inclusion and therefore $U(-t)$ satisfies the Forelli condition, so $t\in\mathscr F_B$ and $\mathscr F_B$ is closed.  
\end{proof}

\begin{definition}
	\label{def2}
	We say that the unitary matrix $B$ is a {\it weighted permutation matrix} if there exists a permutation $\sigma$ of $\{1,\dots, n\}$ such that $|b_{i,\sigma(i)}|=1$ for all $1\leq i\leq n$ (which implies that $b_{i,j}=0$ for all $j\neq \sigma(i)$).
\end{definition}

\begin{theorem}
	\label{thf3}
	$\{U(t)\}$ satisfies the Forelli condition if and only if $B$ is a weighted permutation matrix. 
\end{theorem}

\begin{proof}
	Assume that $B$ is a weighted permutation matrix with permutation $\sigma$. Let $0<t<\lmin=\min_k l_k$. For $g\in L^\infty(\Omega)$ define 
	$$\gamma_t(g)(x)=\left\{\begin{array}{cc}
		g(x+t),&\mbox{if }x\in(\alpha_i,\beta_i-t)\mbox{ for some $i$};\\
		b_{i,\sigma(i)}g(\alpha_{\sigma(i)}+t-(\beta_i-x)),&\mbox{if }x\in(\beta_i-t,\beta_i)\mbox{ for some $i$}.
		\end{array}\right.$$
Let $f\in L^2(\Omega)$.		We check that $U(t)M_gf=M_{\gamma_t(g)}U(t)f$. Let $x\in\Omega$. For a.e. $x\in(\alpha_i,\beta_i-t)$ we have 
$$U(t)M_gf(x)=g(x+t)f(x+t)=M_{\gamma_t(g)}U(t)f(x)$$
For a.e. $x\in(\beta_i-t,\beta_i)$, by Corollary \ref{cor4.5}, we have
$$U(t)M_gf(x)=b_{i,\sigma(i)}g(\alpha_{\sigma(i)}+t-(\beta_i-x))f(\alpha_{\sigma(i)}+t-(\beta_i-x))=M_{\gamma_t(g)}U(t)f(x).$$
Thus $U(t)M_fU(t)^*=M_{\gamma_t(g)}$, and therefore $t\in\mathscr F_B$. Since $\mathscr F_B$ is a closed subgroup of $\br$, it is either a lattice or $\br$, so $\mathscr F_B$ has to be $\br$. 

For the converse, assume that $\{U(t)\}$ satisfies the Forelli condition. Let $0<t<\lmin$. Let $g\in L^\infty(\Omega)$ and $f\in L^2(\Omega)$. With Corollary \ref{cor4.5}, since $U(t)M_gf=M_{\gamma_t(g)}U(t)f$, we have for almost every $x\in(\beta_i-t,\beta_i)$, 
\begin{equation}
	\label{eqf*1}
	\sum_{j=1}^nb_{i,j}f(\alpha_j+t-(\beta_i-x))g(\alpha_j+t-(\beta_i-x))=\gamma_t(g)(x)\sum_{j=1}^nb_{i,j}f(\alpha_j+t-(\beta_i-x)).
\end{equation}
We prove that there can not be two indices $j_1\neq j_2$ with $b_{i,j_1}, b_{i,j_2}\neq 0$. If not, then pick $g\in L^\infty(\Omega)$ such that for  $x\in(\beta_i-t,\beta_i)$, $g(\alpha_{j_1}+t-(\beta_i-x))=1$ and $g(\alpha_j+t-(\beta_i-x))=0$ for $j\neq j_1$. Pick also $f\in L^2(\Omega)$ such that for $x\in (\beta_i-t,\beta_i)$, 
$$f(\alpha_j+t-(\beta_i-x))=\cj b_{i,j}\left(\cj g(\alpha_j+t-(\beta_i-x))-\cj{\gamma_t(g)(x)}\right),\quad (1\leq j\leq n).$$
Plugging these into \eqref{eqf*1} we obtain for a.e. $x\in (\beta_i-t,\beta_i)$,
$$\sum_{j=1}^n|b_{i,j}|^2\left|g(\alpha_j+t-(\beta_i-x))-\gamma_t(g)(x)\right|^2=0.$$
Since $b_{i,j_1},b_{i,j_2}\neq 0$, we must have 
$$\gamma_t(g)(x)=g(\alpha_{j_1}+t-(\beta_i-x))=1\mbox{ and }\gamma_t(g)(x)=g(\alpha_{j_2}+t-(\beta_i-x))=0,$$
		which is a contradiction. 
		
		Thus, there is only one index $j$ such that $b_{i,j}\neq 0$ and we denote it by $\sigma(i)$. Since $B$ is unitary, we must have $|b_{i,\sigma(i)}|=1$. To prove that $\sigma$ is a permutation, we use the same argument as at the end of the proof of Theorem \ref{thm4}.

\end{proof}

\begin{theorem}
	\label{thf4}
	Suppose $\Omega$ is spectral and it satisfies the Forelli condition. Let $L:=\mathfrak m(\Omega)=l_1+\dots+l_n$. Then:
	\begin{enumerate}
		\item $B$ is a weighted permutation matrix for some cycle $\sigma$ of length $n$. There exists $\theta_0\in \br$ such that 
		\begin{equation}
			\label{eqf4.1}
			b_{i,\sigma(i)}=e^{2\pi i\frac{\theta_0}L(\alpha_{\sigma(i)}-\beta_i)},\quad(1\leq i\leq n).
		\end{equation}
		\item All the numbers $\alpha_{\sigma(i)}-\beta_i$ are in $L\bz$. 
		\item $\Lambda=\frac1{L}(\bz-\theta_0)$ and $\Omega$ tiles by $L\bz$. 
		\item The statements in Theorem \ref{thm5}(iii) hold. 
	\end{enumerate}
\end{theorem}

\begin{proof}
	By Theorem \ref{thf3}, $B$ is a weighted permutation matrix for some permutation $\sigma$. The proof that $\sigma$ is a cycle of length $n$ is the same as the proof of Theorem \ref{thm5}(i). 
	
	Let $b_{i,\sigma(i)}=e^{2\pi i\theta_i}$, for $1\leq i\leq n$. For $\lambda\in\Lambda$ we have $Be_\lambda(\vec\alpha)=e_\lambda(\vec\beta)$ and therefore, $b_{i,\sigma(i)}e^{2\pi i\lambda\alpha_{\sigma(i)}}=e^{2\pi i\lambda\beta_i}$, and $e^{2\pi i(\lambda(\alpha_{\sigma(i)}-\beta_i)+\theta_i)}=1$. Therefore 
	$$\Lambda\subseteq \frac{1}{\alpha_{\sigma(i)}-\beta_i}(\bz-\theta_i).$$
	We perform the translations from Theorem \ref{thm5}(iii). Using Lemma \ref{lemmov}, after each translation the new set has the same the spectrum $\Lambda$. At the end, the interval $(\alpha_1,\alpha_1+L)$ has spectrum $\Lambda$. But all the spectra of an interval are known, and we must have some number $\theta_0\in\br$ such that 
	$$\Lambda=\frac1L(\bz-\theta_0).$$
	In particular $\Omega$  also has the spectrum $\frac1L\bz$ and hence it tiles by $L\bz$.
	
	We use again the boundary conditions $Be_\lambda(\vec\alpha)=e_\lambda(\vec\beta)$, for all $\lambda\in\frac{1}{L}(\bz-\theta_0)$ and we obtain 
	$$b_{i,\sigma(i)}e^{2\pi i \frac {k-\theta_0}L\alpha_{\sigma(i)}}=e^{2\pi i\frac {k-\theta_0}\beta_i}\quad (1\leq i\leq n, k\in\bz).$$
	Then 
	$$b_{i,\sigma(i)}e^{-2\pi i\frac{\theta_0}L(\alpha_{\sigma(i)}-\beta_i)}=e^{-2\pi i\frac kL(\alpha_{\sigma(i)}-\beta_i)}.$$
	The left-hand side is independent of $k$, therefore the right-hand side should be as well. Plug in $k=0$ and $k=1$ and obtain $e^{-2\pi i\frac 1L(\alpha_{\sigma(i)}-\beta_i)}=1$, which means that $\alpha_{\sigma(i)}-\beta_i\in L\bz$. Then 
	$$b_{i,\sigma(i)}=e^{2\pi i\frac{\theta_0}L(\alpha_{\sigma(i)}-\beta_i)}.$$
\end{proof}

\begin{theorem}
	\label{thf5}
	Suppose all the intervals of $\Omega$ have the same length $l_k={\ell}$ for all $1\leq k\leq n$ and that $\mathscr F_B\neq\{0\}$. If $t_0\in\mathscr F_B$, $t_0>0$  and $(p-1)\ell< t_0\leq p\ell$, then $B^p$ is a weighted permutation matrix.  
\end{theorem}

\begin{proof}
Let $g\in L^\infty(\Omega)$ and $f\in L^2(\Omega)$. We have the Forelli condition $U(t_0)M_gf=U_{\gamma_{t_0}(g)}U(t_0)f$. Then, with Corollary \ref{cor4.6}, for a.e. $x\in (\alpha_i,\beta_i)$, with $(p-1)\ell<t_0-(\beta_i-x)<p\ell$, setting $r(x)=t_0-(\beta_i-x)-(p-1)\ell$, 
$$\sum_{j=1}^n[B^p]_{i,j}f(\alpha_j+r)g(\alpha_j+r(x))=\gamma_{t_0}(x)\sum_{j=1}^n[B^p]_{i,j}f(\alpha_j+r(x)).$$	
Suppose there are two indices $j_1,j_2$ such that $[B^p]_{i,j_1},[B^p]_{i,j_2}\neq 0$. Pick $g\in L^\infty(\Omega)$ such that $g(\alpha_{j_1}+r(x))=1$ and $g(\alpha_j+r(x))=0$ for $j\neq j_1$, for $x\in(\alpha_i,\beta_i)$ with  $(p-1)\ell<t_0-(\beta_i-x)<p\ell$. Pick also $f\in L^2(\Omega)$ such that for  $x\in(\alpha_i,\beta_i)$ with  $(p-1)\ell<t_0-(\beta_i-x)<p\ell$, 
$$f(\alpha_j+r(x))=\cj{[B^p]_{i,j}}\left(\cj g(\alpha_j+r(x))-\cj{\gamma_{t_0}(x)}\right),\quad(1\leq j\leq n).$$
Plug these into the relation above to obtain 
$$\sum_{j=1}^n|[B^p]_{i,j}|^2\left|g(\alpha_j+r(x))-\gamma_{t_0}(g)(x)\right|^2=0.$$

Since $[B^p]_{i,j_1},[B^p]_{i,j_2}\neq 0$, we obtain for a.e. $x\in (\alpha_i,\beta_i)$ with  $(p-1)\ell<t_0-(\beta_i-x)<p\ell$, 
$$\gamma_{t_0}(g)(x)=g(\alpha_{j_1}+r(x))=1\mbox{ and }\gamma_{t_0}(g)(x)=g(\alpha_{j_2}+r(x))=0,$$
which is the contradiction. The rest is as at the end of the proof of Theorem \ref{thf3}.
\end{proof}

\begin{acknowledgements} The authors are pleased to acknowledge the discussions, the help, the inspiration and the suggestions from Professors Deguang Han, Palle Jorgensen, Mihalis Kolountzakis and ChunKit Lai.
\end{acknowledgements}

\bibliographystyle{alpha}	

\bibliography{eframes}

\end{document}